\documentclass[leqno,10pt]{amsart}

\usepackage{amssymb}
\usepackage{amsmath}
\usepackage[usenames, dvipsnames]{color}
\usepackage{enumitem}
\usepackage{amsthm}
\usepackage{mathrsfs}
\usepackage{amsrefs}
\numberwithin{equation}{section}

\usepackage{marginnote}
\usepackage{todonotes}
\usepackage{hyperref}
\usepackage{esint}
\usepackage{verbatim}

\usepackage{hyperref}

\usepackage{etoolbox}
\theoremstyle{plain}
\newtheorem{theorem}{Theorem}[section]
\newtheorem{lemma}[theorem]{Lemma}
\newtheorem{corollary}[theorem]{Corollary}
\newtheorem{proposition}[theorem]{Proposition}

\theoremstyle{definition}
\newtheorem{definition}[theorem]{Definition}

\newtheorem{remark}[theorem]{Remark}

\newcommand{\R}{\mathbb{R}}
\newcommand{\N}{\mathbb{N}}

\newcommand{\A}{\mathcal{A}}

\def\XXint#1#2#3{{\setbox0=\hbox{$#1{#2#3}{\int}$ }
\vcenter{\hbox{$#2#3$ }}\kern-.6\wd0}}

\title[Harnack inequality singular-degenerate parabolic equations]{Harnack inequality for singular or degenerate parabolic equations in non-divergence form}

\author[S. Cho]{Sungwon Cho}
\address[S. Cho]{Department of Mathematics Education, Gwangju National University of Education, 55 Pilmundaero Buk-gu,
Gwangju, Republic of Korea (61204)}
\email{scho@gnue.ac.kr}

\author[J. Fang]{Junyuan Fang}
\address[J. Fang]{Department of Mathematics, University of Tennessee, 227 Ayres Hall,
1403 Circle Drive, Knoxville, TN 37996-1320 }
\email{jfang9@vols.utk.edu}

\author[T. Phan]{Tuoc Phan}
\address[T. Phan]{Department of Mathematics, University of Tennessee, 227 Ayres Hall,
1403 Circle Drive, Knoxville, TN 37996-1320}
\email{phan@utk.edu}

\subjclass[2020]{35B05, 35B45, 35B65, 35K65, 35K67, 35K10}
\keywords{Krylov-Safonov Harnack inequality, Equations in non-divergence form, Singular degenerate coefficients, Weighted H\"{o}lder estimates, Liouville theorem.}

\thanks{The research of T. Phan was partially supported by Simons Foundation, grant \# 769369. Most part of this work was done when S. Cho visited the Department of Mathematics at the University of Tennessee - Knoxville. S. Cho would like to express his sincere gratitude for the warm hospitality and invaluable support during the visit.}

\begin{document}
%\nocite{*} 
\begin{abstract} This paper studies a class of linear parabolic equations in non-divergence form in which the leading coefficients are measurable and they can be singular or degenerate as a weight belonging to the $A_{1+\frac{1}{n}}$ class of Muckenhoupt weights. Krylov-Safonov Harnack inequality for solutions is proved under some smallness assumption on a weighted mean oscillation of the weight. To prove the result, we introduce a class of generic weighted parabolic cylinders and the smallness condition on the weighted mean oscillation of the weight through which several growth lemmas are established. Additionally, a perturbation method is used and the parabolic Aleksandrov-Bakelman-Pucci type maximum principle is crucially applied to suitable barrier functions to control the solutions. As corollaries, H\"{o}lder regularity estimates of solutions with respect to a quasi-distance, and a Liouville type theorem are obtained in the paper.
\end{abstract}
\maketitle
\section{Introduction} 
\subsection{Settings} This paper proves Krylov-Safonov Harnack inequality for solutions to a class of second order linear parabolic equations in non-divergence form in which the leading coefficients are measurable, but they do not satisfy the uniformly elliptic nor boundedness conditions. The obtained result therefore naturally extends the well-known and classical one due to N. V. Krylov and M. V. Safonov \cite{KrySa-1, KrySa-2}, and it is also the first available result in the research line.  

\smallskip
To set up, let us denote the parabolic operator
\begin{equation} \label{L-def} 
\mathcal{L} u = u_t - \omega(x) a_{ij}(x,t) D_{ij} u,
\end{equation} 
where $a_{ij} = a_{ij}(x,t)$ is assumed to be measurable in $(x,t)$ in its considered domain in $(n+1)$-dimensional Euclidean space,  
namely, 
$\mathbb{R}^{n+1} = \{X=(x,t) :  x=(x_1, x_2, \ldots x_n) \in \mathbb{R}^n, t \in (-\infty, \infty) \} $ for $n \ge 1$. Moreover, $\omega$ is a non-negative measurable function, $D_{ij}$ denotes the second order partial derivatives with respect to $x_i$ and $x_j$ variables for $i, j = 1, 2,\ldots, n$, and  $u_t$ denotes the partial derivative of $u$ in $t$-variable. Also, throughout the paper,  the Einstein summation convention is used.

\smallskip
We assume that the matrix $(a_{ij})$ is symmetric and it satisfies the uniformly elliptic and boundedness conditions: there is $\nu \in (0,1)$ such that
\begin{equation} \label{elliptic}
\nu |\xi|^2 \leq a_{ij}(x,t) \xi_i \xi_j, \quad |a_{ij}(x,t)| \leq \nu^{-1}, \quad \forall \ \xi = (\xi_1, \xi_2,\ldots, \xi_n) \in \mathbb{R}^n
\end{equation}
for every $(x,t)$ in the considered domains. Moreover, we assume that $\omega = \omega(x)$ is in the $A_{1+\frac{1}{n}}$ class of Muckenhoupt weights. As such, the function $\omega$ can vanish and be singular at some points $x$. Consequently,  the leading coefficients $\omega(x) a_{ij}(x,t)$ in the parabolic operator $\mathcal{L}$ can be degenerate,  singular, or both degenerate and singular. When $\omega \equiv 1$,  $\mathcal{L}$ is reduced to a parabolic operator with uniformly elliptic and bounded coefficients.

\smallskip
To briefly describe our results, let us introduce a class of weighted parabolic cylinders. Throughout the paper, for every $\rho>0$ and $x_0\in \mathbb{R}^n$, $B_\rho(x_0)$ denotes the open ball in $\mathbb{R}^n$ of radius $\rho$ centered at $x_0$.  For the given weight $\omega$ in \eqref{L-def}, and for the spatial variable $x$ in the considered domain, let us denote
\begin{equation} \label{mu-def}
\mu(x) = \omega(x)^{-n},
\end{equation}
and
\[
(\mu)_{B_\rho(x_0)} = \frac{1}{|B_\rho(x_0)|} \int_{B_\rho(x_0)} \mu(x) dx,  \quad \text{where} \quad 
 |B_\rho(x_0)| = \int_{B_\rho (x_0)} dx.
\]
Now, for $Y = (y,s) \in \mathbb{R}^{n} \times \mathbb{R}$ and $r>0$, let $C_{r, \mu}(Y)$ be the weighted parabolic cylinder of radius $r$ centered at $Y$, defined by
\begin{equation} \label{intro-cylinder}
C_{r, \mu}(Y) = B_r(y) \times \big (s-r^2(\mu)_{B_r(y)}^{1/n}, s\big).
\end{equation}
\noindent
Here is an informal version of our main result on the Harnack inequality, which gives an answer to the problem posed in \cite{DPT-1} and also in \cite[p. 168]{G-Dong}. See Theorem \ref{Harnack inequality} below for its formal statement.

\begin{theorem}[Harnack inequality] \label{show-off-theorem} Assume that \eqref{elliptic} holds in $C_{2r, \mu}(Y)$ with $r>0$ and $Y=(y, s) \in \mathbb{R}^n \times \mathbb{R}$. Assume also $\omega \in A_{1+\frac{1}{n}}$ and that it satisfies some smallness weighted mean oscillation condition. Then, the following Krylov-Safonov Harnack inequality
\begin{equation} \label{Harnack-show}
\sup_{U_1}u\leq N \inf_{U_2}u
\end{equation}
holds for every non-negative solution $u$ solving the equation 
\begin{equation} \label{show-off-eqn}
\mathcal{L} u =0 \quad \text{in} \quad C_{2r,\mu}(Y).
\end{equation}
Here in \eqref{Harnack-show}, $N$ is a positive generic constant independent of $u$, $r$ and $Y$. Moreover,
\begin{align*}
U_1& =B_{r}(y) \times\big(s-3r^2(\mu)_{B_{2r}(y)}^{1/n}, s-2r^2(\mu)_{B_{2r}(y)}^{1/n}\big), \\
U_2 &=B_{r}(y) \times\big(s-r^2(\mu)_{B_{2r}(y)}^{1/n}, s \big).
 \end{align*}
\end{theorem}

\smallskip
We would like to point out that the class of weighted cylinders $C_{r,\mu}(Y)$ defined in \eqref{intro-cylinder} can be derived from a quasi-distance function $\rho_{\omega}$, which we define in \eqref{quasi-metric-intro} below.  Hence, in some sense, there is a geometric structure hidden in the PDE of $\mathcal{L}  u =0$ that endows the underlying space $\mathbb{R}^{n+1}$ with the inhomogeneous metric $\rho_\omega$. In Corollary \ref{Holder regularity}, weighted H\"{o}lder regularity estimates with respect to $\rho_\omega$ for solutions solving $\mathcal{L} u =0$ are obtained. In addition, a Liouville type theorem for bounded solutions to $\mathcal{L}u =0$ in $\mathbb{R}^n \times (-\infty, 0)$ is proved in Corollary \ref{Liouville theorem} as a consequence of our Harnack estimates.

\smallskip
We also note that in the series of papers \cite{Chia-1, Chia-2, Chia-3}, a class of parabolic equations with singular degenerate coefficients similar to our case but in divergence form was studied. Under some certain conditions, Moser Harnack inequalities were proved in \cite{Chia-1, Chia-3}. Significantly, it was also shown in \cite{Chia-2} that Moser Harnack inequalities are false in certain cases. See also \cite{Bella, Tru-1} for the results in the elliptic case in which Moser Harnack inequalities were obtained. The novelty in Theorem \ref{show-off-theorem} is the introductions of the class of weighted cylinders $C_{r, \mu}(Y)$, and of the smallness condition on a weighted mean oscillation of the weight $\omega$, through which important growth lemmas as in \cite{FeSa} are established. Our results on Harnack inequalities and H\"{o}lder estimates provide the non-divergence counterpart of the classical ones for equations in divergence form  obtained in \cite{Fabes-1, Fabes, Tru-1, Chia-1, Chia-2, Chia-3}. 

\smallskip
As an example, we consider $\omega(x) = |x|^\beta$ for $x \in B_{2r}$, where $B_{\rho} = B_\rho(0)$ for $\rho>0$. Note that $\omega \in A_{1+\frac{1}{n}}$ when $\beta \in (-n, 1)$, and $(\mu)_{B_{2r}}^{1/n} \sim (2r)^{-\beta}$. Then, the equation \eqref{show-off-eqn} with $Y=(0,0)$ is roughly reduced to
\[
u_t - |x|^{\beta} a_{ij}(x,t) D_{ij}u =0 \quad \text{in} \quad B_{2r} \times \big(-(2r)^{2 - \beta}, 0\big).
\] 
Then \eqref{Harnack-show} holds when $|\beta|$ is sufficiently small with
\begin{align*}
U_1& =B_{r} \times\Big(-\frac{3}{4}(2r)^{2-\beta}, -\frac{1}{2} (2r)^{2-\beta}\Big)\quad \text{and}\quad U_2=B_{r} \times\Big(-\frac{1}{4}(2r)^{2-\beta}, 0 \Big).
 \end{align*}
 See the discussion right after the statement of Theorem \ref{Harnack inequality} for more details. Note that in this setting and with the simplest case that $(a_{ij})$ is the identity matrix, our results (Harnack inequalities, H\"{o}der estimates, Liouville theorem) are new. 

\smallskip
We would like to note that it is possible to follow the known techniques such as those in \cite{G-Dong, G-Chen, Krylov-book, Liber} to extend the main results (Theorem \ref{Harnack inequality}, Corollary \ref{Holder regularity}, and Corollary \ref{Liouville theorem}) to a larger class of equations with lower order terms and non-homogeneous right hand side. Namely, our results can be extended to the class parabolic equations
\begin{equation*}
u_t - \omega(x) a_{ij}(x,t) D_{ij} u + b_i(x,t) D_i u + c(x,t) u =f  
\end{equation*}
with the same assumptions on $\omega$ and $(a_{ij})$, and with suitable conditions on $b_i$, $c$, and $f$. However, to avoid unnecessary technical issues and for simplicity in presentation, in this paper, we only study the class of equations $\mathcal{L} u=0$ with $\mathcal{L}$ defined in \eqref{L-def}. 
%=====
\subsection{Relevant Literature} The Harnack inequality  was first formulated and proved by A. Harnack  for harmonic functions when $n=2$, see \cite[paragraph 19, p. 62]{Harnack}. It was then not obvious until 1950s that the analog of the Harnack inequality for non-negative caloric functions, i.e., non-negative solutions $u$ to the heat equation
\[ u_t - \Delta u =0.\]
This parabolic Harnack estimate was first established by J. Hadamard in \cite{Hadamard} and independently by B. Pini in \cite{Pini}.

\smallskip
It is J. Moser who first proved Harnack inequalities for non-negative weak solutions 
to second order linear equations in divergence form with uniformly elliptic and bounded measurable coefficients. See 
 \cite{Moser-1} for elliptic case and \cite{Moser-2}  for parabolic case in which H\"{o}lder regularity estimates of solutions were also proved. Note also that results on H\"{o}lder continuity of solutions without using Harnack inequalities were independently obtained  in the celebrated papers \cite{DeGiorgi, Nash}. Nowadays, the fundamental papers \cite{DeGiorgi, Moser-1, Moser-2, Nash}  all together are known as the De Giorgi - Nash - Moser theory. It is important to note that Moser Harnack inequalities were also proved for a larger class of quasi-linear equations in divergence form, see \cite{Serrin, Trudinger-0, Trudinger}. For further results and developments, see the papers \cite{Ivannov, Serrin-1, Trudinger-2, Mingione, Kassmann} and the books \cite{DiBenedetto, DGV, Gilbarg-Trudinger, Liber}. 

\smallskip
Though, the De Giorgi - Nash - Moser method works well for a large class of equations, it does not apply to equations in non-divergence form with uniformly elliptic and bounded measurable coefficients. The breakthrough was due to N. V. Krylov and M. V. Safonov who first proved Harnack inequalities for this class of equations in the well-known papers \cite{KrySa-1, KrySa-2}. From these celebrated papers, the Krylov-Safonov Harnack inequalities and the method to prove them play a very crucial role in the development of the regularity theory for linear, nonlinear, and fully nonlinear second order elliptic and parabolic equations satisfying the uniformly elliptic and bounded conditions, see the books \cite{CC, Krylov-book, Gilbarg-Trudinger, G-Dong, Liber} for more results and discussion.

\smallskip
Equations with singular or degenerate coefficients appear naturally from both pure and applied mathematics, see \cite{Caffarelli, DPS, DPT-1, EM, JX, Le, Lin-p, Pop-2} and the references therein. Literature on Harnack inequalities and H\"{o}lder estimates for such  equations have been extensively developed in the last few decades.  Below, let us give some references with closely related results. H\"{o}lder regularity estimates for solutions to elliptic equations in divergence form with singular and degenerate coefficients, which are in the $A_2$-Muckenhoupt class, were proved in the classical papers \cite{Fabes-1, Fabes}.  Under some integrability conditions on the smallest and largest eigenvalue functions of the leading coefficient matrix, Moser Harnack inequalities were proved in \cite{Tru-1}. See \cite{Bella} for the recent resolution of the problem on the optimality of the integrability conditions in \cite{Tru-1}. Parabolic equations in divergence form with singular degenerate coefficients were also studied in \cite{Chia-1, Chia-2, Chia-3}, and Moser Harnack inequalities were particularly proved to be false in certain cases. Various similar results were also obtained for degenerate quasi-linear equations in divergence form. For example, see \cite{GW} and the references therein for results and discussion. On the other hand, many other results on Krylov-Safonov Harnack inequalities, H\"{o}lder estimates, and Schauder estimates were also obtained for linear, non-linear equations with some structural singular-degenerate coefficients that are different from our settings. For examples, see \cite{JX,  Le, Lin-p, Pop-2} for results on H\"{o}lder and Schauder estimates, and \cite{Caffarelli, G-Chen, NLe, Maldonado, Krylov-drift, Mooney, Silvestre, Savin} for results on Harnack inequalities and others.

\smallskip
The objective of this work is twofold. On one hand, it provides the Krylov-Safonov Harnack inequality analog  of the Moser one developed in \cite{Bella, Chia-1, Chia-2, Chia-3,  Fabes,Tru-1} for equations in divergence form.  On the other hand, the type of singular-degenerate coefficients as in $\mathcal{L}$ appears in various problems such as geometric analysis, fluid mechanics, and mathematical finance. For examples, see \cite{CMP, DPS, DPT-1, DPT-2, DP, EM, JX} and the references therein. Results on well-posedness and regularity estimates of solutions in Sobolev spaces for such class of equations have been developed recently in \cite{CMP, DPS, DPT-1, DPT-2, DP}. This paper provides a new development in the research program. Applications of our results in the study of nonlinear and fully nonlinear equations will be addressed in our forthcoming papers.
%===
\subsection{Main ideas and approaches}  To prove the Harnack inequality \eqref{Harnack-show}, we follow and adjust the approach introduced in \cite{FeSa}, which is originated and developed from \cite{KrySa-1, KrySa-2, E.M.Landis}. In the approach, it is essential to formulate and prove suitable growth lemmas (see Theorem 3.3, Theorem 4.2, and Theorem 5.3 in \cite{FeSa}). In the setting when coefficients are uniformly elliptic and bounded, the growth lemmas can be viewed as the direct consequences of the heat diffusion phenomena. However, this is not obvious when the coefficients are degenerate, singular, or degenerate and singular as the diffusion can be disrupted and the heat can be disconnected or the heat spreads so quickly in space along time. Our main idea is to apply the perturbation method by freezing the weight $\omega$. The class of weighted cylinders and the weighted mean oscillation of the weight $\omega$ which are generic in preserving the scaling properties are introduced. Through those, we use suitable barrier functions and the Krylov-Tso version of the classical Aleksandrov-Bakelman-Pucci estimates to control the solutions. Important growth lemmas (Proposition \ref{growth-1}, \ref{second-growth-lemma}, and \ref{third-growth}) are formulated and proved. A weighted version of the Krylov-Safonov covering lemma (Lemma \ref{covering-1}) required in the proof is also established.
\smallskip
\subsection{Organization of the paper} The rest of the paper is organized as follows. In the next section, Section \ref{main-result-state}, the main results of the paper with precise definitions and conditions are stated. In Section \ref{Prelim-section}, we recall and prove several preliminary analysis results needed for the paper. In particular, the definitions and important properties of the $A_p$  Muckenhoupt class of weights are recalled.  Properties of the weighted parabolic cylinders are proved. In this section, we also recall several functional spaces and the Krylov-Tso version of the classical Aleksandrov-Bakelman-Pucci theorem.  Moreover, the weighted version of the Krylov-Safonov covering lemma is also formulated. Section \ref{growth-section} is devoted to formulate and prove the growth lemmas that are needed to prove our Harnack inequalities. In particular, three important growth lemmas (Proposition \ref{growth-1}, \ref{second-growth-lemma}, and \ref{third-growth}) and their corollaries are proved. Then, in Section \ref{proof-section}, we provide the proofs of our main results: Harnack inequalities (Theorem \ref{Harnack inequality}), H\"{o}lder continuity (Corollary \ref{Holder regularity}), and Liouville theorem (Corollary \ref{Liouville theorem}). The paper concludes with Appendix \ref{Appendix-A} which provides the proof of the weighted covering lemma (Lemma \ref{covering-1}).
%====
\section{Statements of main results} \label{main-result-state}
%\subsection{Statement of main results} 
Let us begin by introducing some notation and definitions. For a given non-negative locally integrable function $\sigma: \mathbb{R}^n \rightarrow [0,\infty)$, and for each measurable set $E \subset \mathbb{R}^n$,  we denote $\sigma(E)$ the measure of $E$ with respect to the measure $\sigma(x) dx$, i.e.,
\[
\sigma(E) =  \int_{E} \sigma(x) dx .
\]
Also, for a given integrable function $f$ on $E$, we denote the mean of $f$ on $E$ with respect to the Lebesgue measure $dx$ by
\begin{equation*}  %\label{evarage-notation}
 (f)_{E} = \fint_{E} f(x) dx = \frac{1}{|E|}\int_{E} f(x)  dx \quad \text{with} \quad |E| = \int_{E} dx.
\end{equation*}
Throughout the paper, for any $r>0$ and $x \in \mathbb{R}^n$, we denote $B_r(x)$ the open ball in $\mathbb{R}^n$ of radius $r$ and centered at $x$.

\smallskip
Next, let us refer to the definition of the $A_p$ class of Muckenhoupt weights in Definition \ref{A-p-def}. We now recall the definition of bounded mean oscillation with respect to a weight $\omega$ introduced in \cite{Muckenhoupt-Wheeden}.
\begin{definition}  \label{def-mean-os}Let $\omega \in A_{1+\frac{1}{n}}$ and $f$  be locally integrable in a non-empty open set $\Omega\subset \R^n$. For $r>0$ and $y \in \Omega$ such that $B_r(y) \subset \Omega$, let us denote the mean oscillation of $f$ on $B_r(y)$ with respect to the weight $\omega$ by
\begin{equation}\label{WBMO def}
[[f]]_{B_r(y), \omega} = \Big( \frac{1}{\omega(B_r(y))} \int_{B_r(y)} \big|f(x) -(f)_{B_r(y)} \big|^{n+1} \omega(x)^{-n} dx \Big)^{\frac{1}{n+1}} .
\end{equation}
We say $f$ is of \textit{weighted bounded mean oscillation with respect to the weight $\omega$ on $\Omega$} if
\begin{equation}\label{def WBMO}
[[f]]_{\text{BMO}(\Omega, \omega)}=\sup_{B_r(y) \subset \Omega}[[f]]_{B_r(y), \omega}  < \infty.
\end{equation}
\end{definition}
\noindent 
\begin{remark} \label{remark-BMO} 
For $\omega \in A_{1+\frac{1}{n}}$, 
it follows from \cite[Theorem 4]{Muckenhoupt-Wheeden} that $[[f]]_{\text{BMO}(\Omega, \omega)}$ is comparable to
\[
 \sup_{B_r(y) \subset \Omega} \Big(\frac{1}{\omega(B_r(y))} \int_{B_r(y)} \big|f(x) -(f)_{B_r(y)} \big|^q \omega(x)^{1-q}  dx \Big)^{\frac{1}{q}} 
\]
for every $q \in [1, n+1]$, and the latter could be more convenient in some manipulation when $q=1$. However, we use \eqref{WBMO def} as it appears naturally in our paper.
\end{remark}
%\smallskip
Recall the definition of cylinders $C_{r, \mu}(Y)$ in \eqref{intro-cylinder}. We note that as $\omega \in A_{1+\frac{1}{n}}$, it follows that $\mu$ is locally integrable. Therefore, $(\mu)_{B_r(y)}$ is well-defined. Our main result of the paper is the following Krylov-Safonov Harnack inequality.
\begin{theorem}[Harnack inequality]\label{Harnack inequality} For every $K_0\in[1,\infty)$ and $\nu \in (0,1)$, there exist a sufficiently small  number $\epsilon=\epsilon(n,\nu,K_0) \in (0,1)$, and  a number $N = N(n, \nu, K_0)~>~0$ such that the following assertion holds. 
Suppose that  \eqref{elliptic} holds in $C_{2r, \mu}(Y)$ for some $r>0$, $Y = (y,s) \in \mathbb{R}^{n+1}$, and for $\mu$ defined in \eqref{mu-def}. Also, suppose that $\omega\in A_{1+\frac{1}{n}}$ satisfies
\begin{equation} \label{wei-osc-cond}
[\omega]_{A_{1+\frac{1}{n}}}\leq K_0 \quad \text{and}  \quad 
[[\omega]]_{\textup{BMO}(B_{4r}(y), \omega)} \leq \epsilon.
\end{equation} 
Then, for every non-negative $u \in C^{2,1}(C_{2r,\mu}(Y)) \cap C(\overline{{C}_{2r,\mu}(Y)})$ solving 
\begin{equation} \label{PDE-main-theorem}
\mathcal{L}u=0 \quad \text{in} \quad C_{2r, \mu}(Y),
\end{equation}
it holds that
\begin{equation}\label{Harnack}
\sup_{U_1}u\leq N\inf_{U_2}u,
\end{equation}
where
\begin{align*}
U_1& =B_{r}(y)\times\big(s-3r^2(\mu)_{B_{2r}(y)}^{1/n},s-2r^2(\mu)_{B_{2r}(y)}^{1/n}\big)\quad \text{and}\\
 U_2&=B_{r}(y)\times\big(s-r^2(\mu)_{B_{2r}(y)}^{1/n},s \big).
 \end{align*}
\end{theorem}
\noindent
To demonstrate that \eqref{wei-osc-cond} holds for some $\omega$, let us consider the following power weight example
\[
\omega(x) = |x|^\beta, \quad x \in \mathbb{R}^n.
\]
Then, it is well-known that $\omega \in A_{1+\frac{1}{n}}$ when $\beta \in (-n, 1)$. Moreover, it follows from \cite[Lemma A.1, Lemma A.2]{CMP} and Remark \ref{remark-BMO} that there are $K_0 = K_0(n)\geq1$ and $N = N(n)>0$ such that
\[
[\omega]_{A_{1+\frac{1}{n}}} \leq K_0 \quad \text{and} \quad [[\omega]]_{\textup{BMO}(\mathbb{R}^n, \omega)} \leq N |\beta|
\]
for all $\beta$ so that $|\beta| \leq 1/2$. Then, Theorem \ref{Harnack inequality} is applicable for the weight $\omega(x) = |x|^\beta$ when $|\beta|$ is sufficiently small. Again, we stress that Theorem \ref{Harnack inequality} is already new even with this $\omega$ and when $(a_{ij})$ is the identity matrix.
%\smallskip
\begin{remark} \label{remark-0912} We would like to point out the following important points.
\begin{itemize}
\item[\textup{(i)}]  The assumptions \eqref{elliptic} and \eqref{wei-osc-cond} are invariant under the standard scaling, translation, and dilation for the heat equations. More importantly, \eqref{elliptic} and \eqref{wei-osc-cond} are also invariant under the changes of variables $(x,t) \mapsto (x+y, \lambda t +s)$ and $(x,t) \mapsto (\lambda x+ y, t+s)$ for $\lambda>0$. Precisely, if $u$ solves  \eqref{PDE-main-theorem}, then $v(x,t) = u(x+y, \lambda t +s)$ solves
\[
v_t - \tilde{\omega}(x) \tilde{a}_{ij}(x, t) D_{ij} v =0 \quad \text{in} \quad C_{2r, \tilde{\mu}}(0),
\]
where $\tilde{\omega}(x) =\lambda \omega(x+y)$,
\[
 \tilde{a}_{ij}(x,t) = a_{ij}(x +y, \lambda t+s), \quad \text{and} \quad \tilde{\mu}(x) = \frac{1}{\tilde{\omega}(x)^n}.
\]
It is not hard to check that $(\tilde{a}_{ij})$ satisfies \eqref{elliptic} and $\tilde{\omega}$ satisfies \eqref{wei-osc-cond}. Similarly, under the change of variable for $(x,t) \mapsto (\lambda x+y, t +s)$, the same conclusion also holds in which 
\[ \tilde{w}(x) = \lambda^{-2} \omega(\lambda x+y) \quad \text{and} \quad  \tilde{a}_{ij}(x,t) = a_{ij}(\lambda x +y,  t+s) .
\]
\item[\textup{(ii)}] If we use the change of variables $(x,t) \mapsto (x+y, \lambda t + s)$ with $\lambda = (\mu)_{B_{2r}(y)}^{1/n}$, then $(\tilde{\mu})_{B_{2r}(0)} =1$, and the equation \eqref{PDE-main-theorem} is reduced to the equation in the standard heat cylinder of radius $2r$ centered at the origin. %However, we do not apply this change of variables in the paper.
\item[\textup{(iii)}] As we use the perturbation method to prove Theorem \ref{Harnack inequality}, for every $X_0 = (x_0, t_0) \in C_{2r, \mu}(Y)$, and for small $\rho >0$, we freeze $\omega$ and perturb it by $(\omega)_{B_\rho(x_0)}$. See \eqref{Lv estimate} and \eqref{perturb-1-0921} below. We then use the ABP estimates and suitable barrier functions to derive needed estimates on solutions. To this end, we need the smallness condition on the mean oscillation of $\omega$ in \eqref{wei-osc-cond}. It is not clear if this smallness condition can be removed.
\end{itemize}
\end{remark}
\smallskip
Next, we state two corollaries of Theorem \ref{Harnack inequality}.  The first corollary is about the H\"{o}lder regularity estimates of solutions. 
To this end, we introduce the quasi-distance $\rho_{\omega}: \R^{n+1}\times \R^{n+1}\rightarrow [0,\infty)$ that endows the class of parabolic cylinders defined in \eqref{intro-cylinder}. For a given $y\in \R^{n}$, and let $\phi_y : [0, \infty) \rightarrow [0, \infty)$ be defined by 
\begin{equation} \label{phi-y}
\phi_y(\tau)=\tau^2(\mu)_{B_\tau(y)}^{1/n}=\sigma_n^{-1/n} \tau\mu(B_\tau(y))^{1/n}, \quad  \tau \in [0, \infty), 
\end{equation}
where $\sigma_n$ is the Lebesgue measure of the unit ball $B_1 \subset \mathbb{R}^{n}$. 
As $\mu=\omega^{-n}$ and $\omega \in A_{1+\frac{1}{n}}$, 
we see that $\phi_y$ is a continuous and strictly increasing function,
\[
\phi_y(0)=0 \quad \text{and} \quad \phi_y (\tau) \to  + \infty \text{ as } 
\tau \to  + \infty . 
\]
Therefore, its inverse function $\phi_y^{-1}: [0, \infty) \rightarrow [0, \infty)$ exists.  
 
%=====
\smallskip
Now, for any $Y = (y, s)$ and $X=(x,t)$ in $\R^{n+1}$ such that $t\leq s$ and $X \not=Y$, we define
\[
\begin{split}
\bar{\rho}_{\omega}(X, Y)^2 & = \left[\max\big\{|x-y|, \phi_y^{-1}(s-t)  \big\}\right]^2 \\
& \quad +\min\big\{|x-y|^2, (\mu)_{B_{\Theta_Y}(y)}^{-1/n}\cdot(s-t) \big\},
\end{split}
\]
where $\Theta_Y$ is the smallest radius $r>0$ such that $X\in \overline{C_{r,\mu}(Y)}$. 
More specifically, 
\begin{equation}\label{theta-introc}
\begin{split}
\Theta_Y =\Theta_{Y}(X) & = \inf\big\{r>0: X\in \overline{C_{r,\mu}(Y)}\big\} \\
& = \max\big\{|x-y|, \phi^{-1}_y(s-t)\big\} . 
\end{split}
\end{equation}
We define
\begin{equation}   \label{quasi-metric-intro}
 \rho_{\omega}(X, Y) =\left\{
 \begin{array}{ll}
  \bar{\rho}_\omega (X, Y) &  \quad \text{if} \quad X \not=Y  \quad \text{and} \quad t \leq s, \\
    \bar{\rho}_\omega (Y, X) &  \quad \text{if} \quad X \not=Y  \quad \text{and} \quad s < t, \\
    0 & \quad \text{if} \quad X=Y.
\end{array} \right.
\end{equation} 
It turns out that $\rho_\omega$ defines a quasi-metric on $\mathbb{R}^{n+1}$ as pointed out in the following remark, whose proof is just a direct calculation and then skipped.
\begin{remark}\label{quasi-metric lemma}
The function $\rho_{\omega}$ defined in \eqref{quasi-metric-intro}  is a quasi-metric. In particular,
\[
\rho_{\omega}(X,Z)\leq 2\big[\rho_{\omega}(X,Y)+\rho_{\omega}(Y,Z)\big]
\]
for every $X,\ Y,\ Z \in \R^{n} \times \mathbb{R}$. Moreover, it is not hard to verify that
\begin{equation}\label{d-r}
\Theta_Y(X) \leq \rho_{\omega}(X,Y)\leq \sqrt {2} \Theta_Y(X),
\end{equation}
for all $X=(x, t),\ Y=(y, s)\in \R^{n+1}$ with $t\leq s$, where $\Theta_Y(X)$ is defined in \eqref{theta-introc}.
\end{remark}
%The proof of this lemma is provided in Appendix \ref{proof-quasi-distance}.
%see Lemma~\ref{quasi-metric lemma} below. 
Next, for $\alpha \in (0,1]$, and for a function $u$ defined on a non-empty bounded open set $\Omega \subset\mathbb{R}^{n+1}$,  our weighted H\"{o}lder norm $\|u\|_{\mathscr{H}^{\alpha,\alpha/2}_\omega(\Omega)}$ of $u$ with respect to the weight $\omega$ is defined by 
\begin{equation}
\|u\|_{\mathscr{H}^{\alpha,\alpha/2}_\omega(\Omega)}=\|u\|_{L^{\infty}(\Omega)}+[[u]]_{\mathscr{H}^{\alpha,\alpha/2}_{\omega}(\Omega)},
\end{equation}
where
\begin{equation}\label{Holder semi-norm}
[[u]]_{\mathscr{H}^{\alpha,\alpha/2}_{\omega}(\Omega)}=\sup_{\substack{X\neq Z\\ X, Z\in \Omega}}\frac{|u(X)-u(Z)|}{\rho_{\omega}(X,Z)^{\alpha}}.
\end{equation}

\smallskip
We now state our result on the H\"{o}lder continuity of solutions to $\mathcal{L}u=0$, which is a consequence of Theorem \ref{Harnack inequality}.
\begin{corollary}[H\"{o}lder regularity]\label{Holder regularity} 
Under the assumptions in Theorem \ref{Harnack inequality} and for $q \in (0, \infty)$, there exist $\alpha=\alpha(n, \nu, K_0)\in(0,1]$ and $N=N(n,\nu, K_0, q)>0$ such that if $u\in C^{2,1}(C_{2r,\mu}(Y))$ solves $\mathcal{L}u=0$ in $C_{2r, \mu}(Y)$ for some $r>0$ and $Y \in \mathbb{R}^{n+1}$, we have
\begin{equation} \label{L-infty-est-intro}
\|u\|_{L^{\infty}(C_{r, \mu}(Y))} \leq N \Big(\frac{1}{\mu(C_{2r, \mu}(Y))} \int_{C_{2r, \mu}(Y)} |u(x,t)|^q \mu(x)\, dxdt\Big)^{1/q},
\end{equation}
and
\begin{equation}\label{Holder}
[[u]]_{\mathscr{H}^{\alpha,\alpha/2}_\omega(C_{r,\mu}(Y))}\leq N  r^{-\alpha} \Big(\frac{1}{\mu(C_{2r, \mu}(Y))} \int_{C_{2r, \mu}(Y)} |u(x,t)|^q \mu(x)\, dxdt\Big)^{1/q}.
\end{equation}
\end{corollary}

\smallskip
Here is another consequence of Theorem \ref{Harnack inequality}.
\begin{corollary}[Liouville theorem]\label{Liouville theorem}
Under the assumptions in Theorem \ref{Harnack inequality} and 
\[
[[\omega]]_{\textup{BMO}(\R^n, \omega)} \leq \epsilon.
\]
If  $u\in C^{2,1}(\R^n\times (-\infty,0))\cap L^{\infty}(\R^n\times (-\infty,0))$ satisfies $\mathcal {L}u=0$ in $\R^n\times (-\infty,0)$, then $u$ is a constant function.
\end{corollary}
\section{Preliminary analysis} \label{Prelim-section}
\subsection{Muckenhoupt class of weights} 

In this subsection, we recall the definition of the Muckenhoupt class of  \texorpdfstring{$A_p$}{Ap} weights, and point out some of its important properties that are needed in the paper. For  a given non-negative locally integrable function $\sigma: \mathbb{R}^n \rightarrow [0,\infty)$, and for each measurable set $E \subset \mathbb{R}^n$,  we recall that
\begin{equation} % \label{measure notations}
\sigma(E) = \int_{E} \sigma(x) dx \quad \text{and} \quad (\sigma)_E = \frac{\sigma(E)}{|E|},
\end{equation}
where $|E|$ denotes the Lebesgue measure of $E$. Also, recall that $B_\rho(x_0)$ denotes the open ball in $\mathbb{R}^n$ of radius $\rho>0$ and centered at $x_0 \in \mathbb{R}^n$.
\begin{definition} \label{A-p-def} Let $p \in [1, \infty)$ and $\sigma$ be a non-negative locally integrable function defined on $\mathbb{R}^n$. We say that $\sigma$ belongs to $A_p$ if  $[\sigma]_{A_p}< \infty$, where
\begin{align*}
[\sigma]_{A_p} = 
\left\{
\begin{array}{ll} \displaystyle{
\sup_{\substack{\rho > 0 \\ x_0 \in \mathbb{R}^{n}}} 
(\sigma)_{B_\rho(x_0)} (\sigma^{-1/(p-1)})_{B_\rho(x_0)}^{p-1}}  & \text{for} \quad 1 < p < \infty, \\
 \displaystyle{\sup_{\substack{\rho > 0 \\ x_0 \in \mathbb{R}^{n}}} 
(\sigma)_{B_\rho(x_0)} \| \sigma^{-1} \|_{L^{\infty}(B_\rho(x_0))}}  & \text{for} \quad p = 1.
\end{array} \right.
\end{align*}
\end{definition}

\smallskip
We note that for $p\in (1,\infty)$, it follows directly from the definition that  if $\sigma \in A_p$, then the weight $\sigma^{-1/(p-1)}$ is in $A_{p'}$ satisfying 
\[
[\sigma^{-1/(p-1)}]_{A_{p'}}=[\sigma]_{A_p}^{1/(p-1)}, \quad \text{where} \quad p'=p/(p-1).
\]
As such, for $\mu$ defined in \eqref{mu-def}, we have $\mu \in A_{n+1}$ when $\omega \in A_{1+\frac{1}{n}}$ and
\begin{equation}\label{A-p' weight}
[\mu]_{A_{n+1}}=[\omega]_{A_{1+\frac{1}{n}}}^n.
\end{equation}

We also observe that if $\sigma \in A_p$ then $[\sigma]_{A_p} \geq 1$. 
In fact, when $1<p<\infty$, for every $x_0 \in \mathbb{R}^n$ and $\rho>0$, it follows from H\"{o}lder's inequality that 
\begin{align*}
1 & = \frac{1}{|B_\rho(x_0)|} \int_{B_\rho(x_0)} dx =  \frac{1}{|B_\rho(x_0)|} \int_{B_\rho(x_0)} \sigma(x)^{\frac{1}{p}}\sigma(x)^{-\frac{1}{p}}  dx  \\
& \leq \Big(\frac{1}{|B_\rho(x_0)|}\int_{B_\rho(x_0)}\sigma(x) dx \Big)^{\frac{1}{p}} \Big(\frac{1}{|B_\rho(x_0)|} \int_{B_\rho(x_0)}\sigma(x)^{-\frac{1}{p-1}} dx \Big)^{\frac{p-1}{p}} \\
& \leq [\sigma]_{A_p}^{\frac{1}{p}}.
\end{align*}
As a result, we see that for $\omega$ defined in \eqref{L-def} and $\mu$ defined in \eqref{mu-def}, if $\omega \in A_{1+\frac{1}{n}}$, then
\begin{equation} \label{A-1-1/n}
1 \leq  (\omega)_{B_\rho(x_0)} (\mu)_{B_\rho(x_0)}^{1/n}  \leq [\omega]_{A_{1+\frac{1}{n}}}\quad \text{for all}\quad B_\rho(x_0)\subset \R^n.
\end{equation}

\smallskip
We now recall the following result on the doubling property of the $A_p$ class of weights.  For more details, see \cite[Proposition 7.1.5, p. 504]{Grafakos-2} for instance.
\begin{lemma}\label{A-p doubling} Let $\sigma \in A
_p$ with $p \in [1,\infty)$. Then, for every ball $B\subset \R^n$ and every measurable set $A$ with $|A \cap B| \not=0$, we have
\begin{equation} \label{mu-doubling}
\sigma(B) \leq [\sigma]_{A_p} \Big(\frac{|B|}{|A\cap B|}\Big)^{p} \sigma(A\cap B).
\end{equation} 
\end{lemma}
From Lemma \ref{A-p doubling} and \eqref{A-p' weight}, we see that when $\omega\in A_{1+\frac{1}{n}}$ with $[\omega]_{A_{1+\frac{1}{n}}}\leq K_0$ for some $K_0 \in [1, \infty)$, then $\mu \in A_{n+1}$, $[\mu]_{A_{n+1}}=[\omega]_{A_{1+\frac{1}{n}}}^n\leq K_0^n$, and 
\begin{equation}\label{doubling property}
\mu(B_{2R}(x_0)\big)\leq2^{n(n+1)}K_0^n\mu\big(B_R(x_0)\big), \quad \forall \ x_0\in \R^n \quad \text{and} \quad \forall \ R >0.
\end{equation}

%See \cite[(8) of Proposition 9.1.5, p. 283-284]{Grafakos-1}. 
\smallskip
The following well-known lemma is a consequence of the reverse H\"{o}lder 
property of the $A_p$ weights. 
For example, see \cite[Proposition 7.2.8, p. 521]{Grafakos-2} for details.
\begin{lemma}\label{reverse property} Let $K_0 \in [1, \infty)$ and $p \in [1, \infty)$, then there exist $\lambda_0 = \lambda_0(n, p, K_0)>0$ and $N = N(n, p, K_0)>0$ such that for every $\sigma \in A_p$ satisfying $[\sigma]_{A_p} \leq K_0$, it holds that
\begin{equation*}
\frac{\sigma(S)}{\sigma(B)}\leq N\Big(\frac{|S|}{|B|}\Big)^{\lambda_0}
\end{equation*}
for any ball $B$ and any measurable set $S \subset B$.
\end{lemma}
\subsection{Weighted parabolic cylinders} \label{cylinders-def-ses}  Let $\mu: \mathbb{R}^n \rightarrow \mathbb{R}$ be a non-negative locally integrable function. For any $Y = (y, s) \in \mathbb{R}^{n+1}$ and $r > 0$, let us recall that $C_{r,\mu}(Y)$ is the weighted parabolic cylinder of radius $r$ centered at $Y$ with respect to the weight $\mu$ and it  is defined by
\begin{equation}\label{cylinder def-1}
\begin{split}
C_{r,\mu}(Y) & = B_r(y) \times \big(s-r^2(\mu)_{B_r(y)}^{1/n}, s\big) \\
& = B_r(y) \times \big(s- \sigma^{-1/n}_n r \mu({B_r(y)})^{1/n}, s\big) ,
\end{split}
\end{equation}
where $B_r(y)$ is the open ball in $\mathbb{R}^n$ of radius $r$ centered at $y$, $\sigma_n$ is the Lebesgue measure of the unit ball $B_1 \subset \mathbb{R}^{n}$, and
\[
 (\mu)_{B_r(y)} = \frac{1}{|B_r(y)|} \int_{B_r(y)} \mu(x) dx.
\]
Note that though the weight $\mu$ only depends on $x$-variable, frequently in the paper, for a measurable set $E \subset \mathbb{R}^{n+1}$, we also denote
\[
\mu(E) = \int_{E} \mu(x)\, dxdt.
\]
It then follows immediately from \eqref{cylinder def-1} that
\begin{equation}\label{cylinder measure formula}
\begin{split}
& r^2(\mu)_{B_r(y)}^{1/n}=\sigma_n^{-1/n}r\mu(B_{r}(y))^{1/n}, \quad \text{and} \quad \\
& \mu(C_{r, \mu}(Y))= \sigma_n^{-1/n} r \mu(B_r(y))^{(n+1)/n}.
\end{split}
\end{equation}
These two formulas are frequently used in the paper. Moreover, throughout the paper, when $Y = (0, 0) \in \mathbb{R}^{n+1}$, for simplicity, we omit the centers and write
\[
B_r = B_r(0), \quad \text{and} \quad C_{r,\mu} = C_{r,\mu}(0).
\]

\smallskip
Next, we recall the following definition of the parabolic boundary for a general domain, see \cite{FeSa} for instance. 
\begin{definition} Let $\Omega$ be an open connected set in $\R^{n+1}$. The \textit{parabolic boundary $\partial_p \Omega$} of $\Omega$ is the set of all points $X_0=(x_0,t_0)\in \partial \Omega$ such that there exists a continuous function $x=x(t)$ on an interval $[t_0, t_0+\delta)$ with values in $\R^n$, such that 
\[
x(t_0)=x_0,\quad \text{and}\quad \big(x(t),t\big)\in \Omega\quad \text{for all} \quad t\in(t_0,t_0+\delta).
\]
Here $x=x(t)$ and $\delta>0$ depend on $X_0$. 
\end{definition}
In particular, we have 
\[
\partial_p C_{r, \mu}(Y) = \Big(\overline{B_r(y)} \times \big\{s-r^2(\mu)_{B_r(y)}^{1/n}\big\} \Big) \cup \Big(\partial B_r(y) \times \big(s-r^2 (\mu)_{B_r(y)}^{1/n}, s\big)\Big).
\]

\smallskip
We now conclude this subsection with the following inclusion property for weighted cylinders.
\begin{lemma}\label{cylinder-inclusion} 
Assume that $\mu$ is a non-negative locally integrable function. Then for every $r>0$, $Y = (y,s)\in \mathbb{R}^{n+1}$, and $\theta \in (0,1)$, it holds that
\[   C_{\theta r, \mu}(X_0) \subset C_{r, \mu}(Y),  \quad \forall \   X_0 = (x_0,t_0) \in C_{(1-\theta)r, \mu}(Y).
\]
\end{lemma}
\begin{proof} Let \(X_0 = (x_0, t_0) \in  C_{(1-\theta)r, \mu}(Y) \) be fixed. Note that by the definition and property of the weighted cylinders in \eqref{cylinder measure formula}, it follows that
\begin{equation} \label{cylinder def-2}
\begin{split}
C_{\theta r,\mu}(X_0) &= B_{\theta r}(x_0) \times \big(t_0 - \sigma_n^{-1/n} \theta r\mu(B_{\theta r}(x_0))^{1/n}, t_0\big), \notag \\
C_{(1-\theta) r,\mu}(Y) &= B_{(1-\theta) r}(y) \times \big(s - \sigma_n^{-1/n} (1-\theta)r\mu(B_{(1-\theta)r}(y))^{1/n}, s\big), \notag \\
C_{r,\mu}(Y) &= B_{r}(y) \times \big(s - \sigma_n^{-1/n} r\mu(B_{r}(y))^{1/n}, s\big), 
\end{split}
\end{equation}
where $\sigma_n$ is the Lebesgue measure of the unit ball in $\mathbb{R}^n$. 

\smallskip
Then, as $X_0 \in C_{(1-\theta)r, \mu}(Y)$, it follows that
\begin{equation} \label{X-zero-in}
x_0 \in B_{(1-\theta)r}(y) \quad \text{and} \quad  s - \sigma_n^{-1/n} (1-\theta)r\mu(B_{(1-\theta)r}(y))^{1/n} < t_0 < s.
\end{equation} 
By the triangle inequality, it follows that \( B_{\theta r}(x_0) \subset B_{r}(y) \). 
Since $t_0 < s$ due to \eqref{X-zero-in},  
it follows that we only need to verify that
\begin{equation}\label{time inclusion}
s - \sigma_n^{-1/n} r\mu(B_{r}(y))^{1/n} \leq t_0 - \sigma_n^{-1/n} \theta  r\mu(B_{\theta r}(x_0))^{1/n}.
\end{equation}
Note that by the left inequality in the second assertion in \eqref{X-zero-in}, we have
\[
t_0 - \sigma_n^{-1/n} \theta  r\mu(B_{\theta r}(x_0))^{1/n} > s - \sigma_n^{-1/n} r\Big\{(1-\theta) \mu(B_{(1-\theta)r}(y))^{1/n} + \theta  \mu(B_{\theta r}(x_0))^{1/n}\Big\}.
\]
Hence, \eqref{time inclusion} follows as
\[
(1-\theta) \mu(B_{(1-\theta)r}(y))^{1/n} + \theta  \mu(B_{\theta r}(x_0))^{1/n} \leq  \mu(B_{r}(y))^{1/n},
\]
because both $B_{(1-\theta)r}(y)$ and $B_{\theta r}(x_0)$ are subsets of $B_{r}(y)$.
\end{proof}
%The proof of this lemma is provided in Appendix \ref{proof-quasi-distance}.
%=====
\subsection{Function spaces and parabolic Aleksandrov-Bakelman-Pucci estimates} Let $q \in [1, \infty)$,  $\Omega \subset \mathbb{R}^{n+1}$ be an open domain, and let  $\sigma$ be a locally integrable function defined on $\Omega$. The weighted Lebesgue space $L^q(\Omega, \sigma)$ is defined by
\[
L^q(\Omega, \sigma) = \Big\{g: \Omega \rightarrow \mathbb{R} \quad \text{measurable}\ : \|g\|_{L^q(\Omega, \sigma)} <\infty \Big\},
\]
where
\[
\|g\|_{L^{q}(\Omega, \sigma)} = \left( \int_{\Omega} |g(x,t)|^q \sigma(x,t) dxdt\right)^{1/q}
\]
for a given measurable function $g: \Omega \rightarrow \mathbb{R}$. We also denote by $L^q_{\text{loc}}(\Omega, \sigma)$ the set of all measurable functions $g: \Omega \rightarrow \mathbb{R}$ such that  $\|g\|_{L^q(K, \sigma)} <\infty$ for every compact set $K \subset \Omega$. 

\smallskip
We note that the space $L^q(\Omega, \sigma)$ is reduced to the standard Lebesgue space $L^q(\Omega)$ when  $\sigma =1$. We also denote by $W^{2,1}_q(\Omega)$ the set of all measurable functions $g \in L^q(\Omega)$ such that its weak derivatives $g_t, D_{i}g, D_{ij} g$ exist and they are in $L^q(\Omega)$ for all $i, j=1,2,\ldots, n$. In a similar definition, we have the space $W^{2,1}_{q, \text{loc}}(\Omega)$, the set of all measurable functions $g:\Omega \rightarrow \mathbb{R}$ such that $g \in W^{2,1}_{q}(K)$ for every compact set $K \subset \Omega$.

\smallskip
Throughout the paper, we also denote by $C(\Omega)$ the space of all continuous functions $g: \Omega \rightarrow \mathbb{R}$. Similarly, $C^{2,1}(\Omega)$ is the space of all $g \in C(\Omega)$ such that all $g_t, D_i g, D_{ij}g$ are in $C(\Omega)$ for all $i, j=1,2,\ldots, n$.

\smallskip
The following well-known Aleksandrov-Bakelman-Pucci type estimates in the parabolic setting were first proved by N. V. Krylov in \cite{Krylov-5}, and then independently by K. Tso in \cite{Tso-10}. The version below can be found in \cite[Theorem 5, p. 34]{Krylov-book} and \cite[Theorem 7.1, p. 156]{Liber}.
\begin{theorem}[ABP estimates]\label{ABP theorem}  
Let $\nu \in (0,1)$, and $\Omega\subset \R^{n+1}$. There exists a constant $N = N(n, \nu)>0$ such that the following assertion holds. 
Assume that \eqref{elliptic} holds in $\Omega$, and suppose there is $R>0$ such that
\[ 
x \in B_R, \quad \forall \ (x,t) \in \Omega.
\] 
Then, for every  $u\in W^{2,1}_{n+1, loc}(\Omega) \cap C(\overline{\Omega})$ satisfying
\[
\mathcal L u \leq f\quad \text{in } \Omega,\quad \text{and}\quad u\leq 0\quad \text{on } \partial_{p}\Omega,
\]
it holds that
\begin{equation*} %\label{WABP}
\sup_{\Omega} u^{+} \leq NR^{n/(n+1)}||f^+||_{L^{n+1}(\Omega^+, \mu)},
\end{equation*}
where $\Omega^+=\Omega\cap \{u>0\}$.
\end{theorem}
As a simple corollary, we obtain the maximum principle for the
parabolic operator $\mathcal{L}$, which will be used in the paper.
\begin{corollary} \label{maximum principle}
In addition to the assumptions of Theorem~\ref{ABP theorem}, 
we assume that $f=0$, namely $ \mathcal L u \leq 0$, then 
we have  $u \leq 0$ on $\Omega$.
%\[
%\sup_{\Omega} u \leq \sup_{\partial_{p}\Omega }u. 
%\]
\end{corollary}
\subsection{Covering lemma} Let $\mu$ be defined as in \eqref{mu-def}, and $Y=(y,s)\in \R^{n+1}$. For a given measurable set $\Gamma\subset \mathbb{R}^{n+1}$ and a fixed constant $\delta_0 \in (0,1)$, let us define
\begin{equation}\label{cylinders set}
\mathcal{A}  =\Big\{C = C_{r, \mu}(Y): \mu(C \cap \Gamma) \geq (1-\delta_0) \mu(C) \Big\},
\end{equation}
and 
\begin{equation}\label{set E}
E =  \displaystyle{\bigcup_{C \in \mathcal{A}} C}.
\end{equation}
Also, for a given number $K_1>1$, for each $C = C_{r, \mu}(Y) \in \mathcal{A}$,  we define
\begin{equation}\label{hat cylinder}
\hat{C}=\hat{C}_{r,\mu}(Y) =B_r(y)\times \big(s +r^2 (\mu)_{B_r(y)}^{1/n}, s + K_1 r^2 (\mu)_{B_r(y)}^{1/n}\big).
\end{equation}
Let us denote
\begin{equation}\label{set hat E}
\hat{E} = \bigcup_{C \in \A} \hat{C}.
\end{equation}
We note that both $\mathcal{A}$ and $E$ depend on $\Gamma$, $\delta_0$, and $\mu$. However, for simplicity in writing, we do not put those factors into $\mathcal{A}$ and $E$. Similarly, $\hat{E}$ also depends on $\Gamma, \delta_0, \mu$, and $K_1$. Moreover, it is also notable to see that $E$ and $\hat E$ defined above are open and measurable due to the opennesses and measurability of $C\in \A$ and the corresponding $\hat C$.

\smallskip
Now we state a covering lemma, whose unweighted case (i.e. $\mu=1$) was proved by N. V. Krylov and M. V. Safonov \cite{KrySa-2}.  See also \cite[Lemmas 5.1-5.2]{FeSa} and \cite[Lemma 9.2.6, p. 173]{Krylov-book} for the unweighted version of the lemma and \cite{Safonov} for the elliptic case.  

\begin{lemma}[Covering Lemma] \label{covering-1} Let $K_0\in[1,\infty)$ and $\delta_0 \in (0,1)$. Suppose that  $\Gamma \subset \mathbb{R}^{n+1}$ is a bounded measurable set, and $\omega \in A_{1+\frac{1}{n}}$ satisfies $[\omega]_{A_{1+\frac{1}{n}}} \leq K_0$. Then, for the set $E$ defined in \eqref{set E}, it holds that
\begin{equation}\label{conclusion covering}
\mu(\Gamma\setminus E) =0 \quad \text{and} \quad 
\mu(E) \geq q_0 \mu(\Gamma)
\end{equation}
for $q_0 = 1 +3^{-(n+1)^2-1}K_0^{-n-1} \delta_0$. Moreover, for the set $\hat E$ defined \eqref{set hat E} with some given $K_1>1$, we also have
\begin{equation}\label{hat-E-E}
\mu\big(\hat{E}\big) \geq q_1 \mu(E) \quad \text{for} \quad q_1 = (K_1-1)(K_1+1)^{-1}.
\end{equation}
%where $q_1 = (K_1-1)(K_1+1)^{-1}>0$.
\end{lemma} 
%===
\noindent
The proof of the lemma follows as that of the un-weighted case with some modifications adjusted to the weighted setting. Since the proof is rather long and it has not appeared anywhere, we provide the details in Appendix \ref{Appendix-A}. 
%====
\section{Growth Lemmas} \label{growth-section}
\subsection{First Growth Lemma}
In this subsection, we state and prove our first growth lemma and its corollary. The following proposition is the main result of the subsection.
%===
\begin{proposition}[First Growth Lemma] \label{growth-1} For  every $K_0 \in [1,\infty)$ and $\nu \in (0,1)$, there exists a constant $N = N(n, \nu, K_0)>0$ such that the following assertion holds. 
Suppose that \eqref{elliptic} holds in $C_{r, \mu}(Y)$ with $\mu$ defined in \eqref{mu-def} for some $r>0$ and $Y \in \mathbb{R}^{n+1}$. Suppose also that $\omega \in A_{1+\frac{1}{n}}$ and it satisfies 
\begin{equation}  \label{growth-1-assum}
 [\omega]_{A_{1+\frac{1}{n}}} \leq K_0 \quad \text{and} \quad [[\omega]]_{\textup{BMO}(B_r(y), \omega)} \leq\epsilon_0  \quad \text{for some} \quad \epsilon_0 \in(0,1). 
\end{equation} 
Then, for every $u \in C^{2,1}(C_{r, \mu}(Y)) \cap C(\overline{{C}_{r, \mu}(Y)})$ satisfying
\[ \mathcal{L} u \leq 0  \quad \text{in} \quad  C_{r, \mu}(Y),
\] 
and
\begin{equation} \label{growth-1-cond}
\mu\big(C_{r, \mu}(Y)\cap\{ u>0\}\big) \leq \delta_0 \mu(C_{r, \mu}(Y))
\end{equation}
for some $\delta_0\in(0,1)$, it holds that
\begin{equation} \label{growth-1-est}
\sup_{C_{r/2, \mu}(Y)} u^+ \leq \min\Big\{N\big(\delta_0^{\frac{1}{n+1}}+\epsilon_0\big), 1\Big\} \sup_{C_{r, \mu}(Y)} u^+.
\end{equation}
\end{proposition}
%===== 
\begin{proof}  Our main idea is to freeze the weight $\omega$ in $\mathcal{L}$, and then apply the ABP estimates to a suitable barrier function.  By using the linear translation $X \mapsto X-Y$, we assume without loss of generality that $Y=0$. Due to this,  we write $C_{r,\mu}(Y)=C_{r,\mu}$ and $B_r(y)=B_r$. 

\smallskip
We begin by proving the claim that there exists $N = N(n, \nu)>0$ such that
\begin{equation} \label{claim-1}
u(0) \leq   \min\Big\{N K_0 \left(\delta_0^{1/(n+1)} + \epsilon_0 \right), 1\Big\}\sup_{C_{r, \mu}} u^+.
\end{equation}
To prove the claim, let us denote $M =\displaystyle{\sup_{C_{r, \mu}} u^+\geq 0}$. It suffices to verify the claim when
\[
M >0 \quad \text{and} \quad u(0) >0.
\]
Let
\[
v(x,t) = u(x,t) - \frac{M}{r^2}\Big(|x|^2 - (\mu)_{B_r}^{-1/n}t\Big) \quad \text{for} \quad (x,t) \in \overline{C_{r, \mu}},
\]
and
\[
\Omega^+=C_{r, \mu}\cap\{ v >0\} .
\]
%Note that $v(0) = u(0)>0$. Hence, $0\in \Omega$ and $\Omega \not=\emptyset$. On the other hand, 
As $v \leq u$ in $C_{r,\mu}$, we see that $\Omega^+ \subset C_{r, \mu}\cap\{ u>0\}$. This, together with \eqref{growth-1-cond}, implies that
\begin{equation} \label{Q-measure}
\mu(\Omega^+) \leq \mu \big(C_{r,\mu} \cap \{ u>0\}\big) \leq \delta_0 \mu(C_{r, \mu}).
\end{equation}
Now, by a direct calculation, using the assumption that $\mathcal{L} u\leq 0$ in $C_{r,\mu}$ and \eqref{elliptic}, we have
%\begin{equation}\label{Lv estimate}
\begin{align}  \notag
\mathcal{L} v & = \mathcal{L} u + Mr^{-2} \big((\mu)_{B_r}^{-1/n} + 2 \omega(x) \text{tr}(a_{ij})\big)\\ \notag
& \leq Mr^{-2}\big((\mu)_{B_r}^{-1/n} + 2n \nu^{-1} \omega(x)\big)\\ \notag
& = Mr^{-2} \big((\mu)_{B_r}^{-1/n} + 2n \nu^{-1} (\omega)_{B_r}\big)  + 2n \nu^{-1}M r^{-2} \big(\omega(x) - (\omega)_{B_r}\big)\\ \label{Lv estimate}
&\leq N Mr^{-2}\big[(\mu)_{B_r}^{-1/n} + (\omega)_{B_r}  +  g(x)\big] \qquad \text{in} \quad  C_{r,\mu},
\end{align}
%\end{equation}
where $N = N(n,\nu)>0$, and 
\[ g(x) =\omega(x) - (\omega)_{B_r}.
\]
Also, due to \eqref{A-1-1/n}, we see that $(\mu)_{B_r}^{-1/n}  \leq (\omega)_{B_r}$. Therefore, it follows from \eqref{Lv estimate} that
\begin{equation*} \label{Lv estimate-2}
\mathcal{L} v  \leq N Mr^{-2}\big[(\omega)_{B_r}  +  g(x)\big] \qquad \text{in} \quad  C_{r,\mu}.
\end{equation*}
We also note that 
\[
v(x,t) \leq M -Mr^{-2}\Big(|x|^2 -(\mu)_{B_r}^{-1/n}t\Big)  \leq 0 \quad \text{for} \quad (x,t) \in \partial_p C_{r, \mu}.
\]
Then, applying the ABP estimates, Theorem \ref{ABP theorem}, to $v$ in $C_{r,\mu}$, we obtain
\begin{align} \notag
\sup_{C_{r,\mu}} v  & \leq   N r^{n/(n+1)}\| (\mathcal{L} v)^+\|_{L^{n+1}(\Omega^+, \mu)}  \\ \label{sup-v-05-13}
& \leq N r^{-(n+2)/(n+1)}M\Big[(\omega)_{B_r} \mu(\Omega^+)^{1/(n+1)} + \|g\|_{L^{n+1}(\Omega^+, \mu)} \Big], \end{align}
where $N = N(n, \nu)>0$.

\smallskip
Next, we control the two terms on the right hand side in \eqref{sup-v-05-13}. 
We note that by \eqref{cylinder measure formula}, \eqref{Q-measure}, and the fact that $[\omega]_{A_{1+\frac{1}{n}}}\leq K_0$, it follows that
\begin{align} \notag
r^{-(n+2)/(n+1)} (\omega)_{B_r} \mu(\Omega^+)^{1/(n+1)} & \leq r^{-(n+2)/(n+1)} (\omega)_{B_r} \mu(C_{r, \mu})^{1/(n+1)}  \delta_0^{\frac{1}{n+1}} \\ \notag
& = N(n) (\omega)_{B_r} (\mu)_{B_r}^{1/n}\delta_0^{1/(n+1)}\\ \label{part-1}
&\leq N(n)K_0\delta_0^{1/(n+1)}.
\end{align}
On the other hand, note that
\begin{align*}
 \| g\|_{L^{n+1}(\Omega^+, \mu)} & \leq \Big( \int_{C_{r,\mu}} |\omega(x) -(\omega)_{B_r}|^{n+1} \omega(x)^{-n}dx dt\Big)^{1/(n+1)} \\
& = r^{2/(n+1)} (\mu)_{B_r}^{1/(n^2+n)}\Big(\int_{B_{r}} |\omega(x) -(\omega)_{B_r}|^{n+1}\omega(x)^{-n}\, dx \Big)^{1/(n+1)} \\
& = r^{2/(n+1)} (\mu)_{B_r}^{1/(n^2+n)} \omega(B_r)^{1/(n+1)}[[\omega]]_{B_r, \omega},
\end{align*}
where $[[\omega]]_{B_r, \omega}$ used in the last step is  defined in \eqref{WBMO def}. 
From this, and the assumption \eqref{growth-1-assum}, it follows that
\begin{align} \notag
r^{-(n+2)/(n+1)} \| g\|_{L^{n+1}(\Omega^+, \mu)} & \leq   r^{-n/(n+1)}(\mu)_{B_r}^{1/(n^2+n)} \omega(B_r)^{1/(n+1)} \epsilon_0 \\ \notag
& = N(n)\big[ (\mu)_{B_r}^{1/n} (\omega)_{B_r}\big]^{1/(n+1)} \epsilon_0 \\ \label{part-2}
& \leq N(n)  K_0^{1/(n+1)}\epsilon_0.
\end{align}
Now, we put the estimates \eqref{part-1} and \eqref{part-2} into \eqref{sup-v-05-13} to obtain 
\begin{align*}
\sup_{C_{r,\mu}} v  \leq N  K_0\big(\delta_0^{1/(n+1)}  +\epsilon_0 \big)M,
\end{align*}
where $N = N(n, \nu)>0$, and we also used the fact that $K_0^{1/(n+1)} \leq K_0$ as $K_0 \geq 1$. As a consequence, we have
\[
u(0) = v(0) \leq  \min\Big\{N K_0 \big(\delta_0^{1/(n+1)} + \epsilon_0 \big), 1\Big\}M  
\]
for $N = N(n, \nu) >0$, and the claim \eqref{claim-1} is proved.  

\smallskip
We now proceed to prove \eqref{growth-1-est}. Let $X_0 =(x_0, t_0) \in C_{r/2, \mu}$ be arbitrary, it follows from Lemma \ref{cylinder-inclusion} that
\[
\{ u>0 \} \cap C_{r/2, \mu}(X_0) \subset \{ u>0 \} \cap C_{r, \mu}.
\]
Using this inclusion, along with \eqref{growth-1-cond} and \eqref{cylinder measure formula},  we obtain
\begin{align} \notag
\mu\left(\{ u>0 \} \cap C_{r/2, \mu}(X_0)\right) & \leq \mu \left(\{ u>0 \} \cap C_{r, \mu}\right) \leq \delta_0 \mu(C_{r, \mu}) \\  \label{part-3}
&=\delta_0 \sigma_n^{-1/n} r \mu(B_r)^{(n+1)/n}.
\end{align}
Noting that  $\omega \in A_{1+\frac{1}{n}}$, it follows from \eqref{A-p' weight} that $\mu \in A_{n+1}$ with $[\mu]_{A_{n+1}} = [\omega]_{A_{1+\frac{1}{n}}}^n\leq K_0^n$.  From \eqref{doubling property}, or by applying Lemma 
\ref{A-p doubling} to $\mu$ with $B= B_{r}$ and $A=B_{r/2}(x_0) \subset B$, we get 
\begin{equation*}
 \mu(B_r) \leq 2^{n(n+1)}K_0^n\mu(B_{r/2}(x_0)).
\end{equation*}
From this, \eqref{part-3}, and \eqref{cylinder measure formula}, it follows that
\begin{align*}
\mu(\{ u>0 \} \cap C_{r/2, \mu}(X_0))  & \leq 2^{(n+1)^2}K_0^{n+1}\delta_0 \sigma_n^{-1/n} r\mu(B_{r/2}(x_0))^{(n+1)/n}\\
%& = 2^{(n+1)^2+1}K_0^{n+1}\delta_0\sigma_n^{-1/n} \big(\frac{r}{2}\big)\mu(B_{r/2}(x_0))^{(n+1)/n} \\
& = 2^{(n+1)^2+1}K_0^{n+1}\delta_0 \mu(C_{r/2, \mu}(X_0)).
\end{align*}
Then, by taking a linear translation and applying the claim \eqref{claim-1} with $X_0$ in place of $0$,  $r/2$ in place of $r$ and $\min \Big\{2^{(n+1)^2+1} K_0^{n+1}\delta_0,1\Big\}$ in place of $\delta_0$, we get
\[
u(X_0)  \leq \min\Big\{N\big(\delta_0^{1/(n+1)} + \epsilon_0 \big), 1\Big\}\sup_{C_{r/2, \mu}(X_0)} u^+,
\]
where $N = N(n,\nu, K_0)>0$. 
Since $X_0$ is arbitrary in $C_{r/2, \mu}$, it follows from the last estimate and  Lemma \ref{cylinder-inclusion} that
\begin{align*}
\sup_{C_{r/2, \mu}} u^+\leq \min\Big\{N\big(\delta_0^{1/(n+1)} + \epsilon_0 \big), 1\Big\}\sup_{C_{r, \mu}} u^+ 
\end{align*}
for $N = N(n, \nu, K_0)>0$. Hence, \eqref{growth-1-est} is proved.
\end{proof}

As a result of Proposition~\ref{growth-1}, we have the following corollary on the upper bound estimates of the subsolutions to $\mathcal{L}u =0$.
\begin{corollary} \label{cor-1} Let $K_0 \in [1, \infty)$ and $\nu \in (0,1)$. There exists $\overline{\epsilon}_0 =\overline{\epsilon}_0(n, \nu, K_0) \in (0,1)$ such that the following assertion holds. Suppose that \eqref{elliptic} holds in $C_{r, \mu}(Y)$ with some $r>0$ and $Y =(y,s) \in \mathbb{R}^{n+1}$. Assume also that  $\omega \in A_{1+\frac{1}{n}}$ satisfies
\begin{equation} \label{assum-cor-grow-1}
[\omega]_{A_{1+\frac{1}{n}}} \leq K_0 \quad \text{and} \quad [[\omega]]_{\textup{BMO}(B_{r}(y), \omega)} \leq \overline{\epsilon}_0. 
\end{equation}
Then, for every $u \in C^{2,1}(C_{r, \mu}(Y)) \cap C(\overline{C_{r, \mu}(Y)})$ satisfying
\[
\mathcal{L} u \leq 0 \quad \text{in} \quad C_{r, \mu}(Y),
\]
it holds that
\begin{equation}\label{bdd esti}
u^+(Y)\leq N \left(\frac{1}{\mu(C_{r, \mu}(Y))} \int_{C_{r, \mu}(Y)} u^+(x,t)^q \mu(x)\, dxdt\right)^{1/q}
\end{equation}
for  $q \in (0, \infty)$, and $N = N(n, \nu, K_0, q)>0$.
\end{corollary}
\begin{proof}  We observe that we only need to prove \eqref{bdd esti} when $q=1$ as the case $q>1$ follows from the case $q=1$ by applying H\"{o}lder's inequality, meanwhile the case $q \in (0,1)$ follows from the case $q=1$ and the standard iteration, see   \cite[p. 75-76]{Han-Lin} for example. From now on, we assume $q=1$.

\smallskip
We follow the idea introduced in \cite[Theorem 3.4]{FeSa}. As $\omega \in A_{1+\frac{1}{n}}$, we have $\mu \in A_{n+1}$ with $[\mu]_{A_{n+1}} = [\omega]_{A_{1+\frac{1}{n}}}^n \leq K_0^n$. Then, by Lemma~\ref{A-p doubling} with $B = B_r (y)$ 
and $A=B_\rho (x) \subset B$, we obtain 
\begin{equation*} %\label{doubling-r-rho}
\mu(B_r (y))  \leq  \Big(\frac{r}{\rho}\Big)^{n(n+1)} K_0^n \mu(B_{\rho}(x))   
\end{equation*}
for all $r \geq \rho>0$ such that $B_\rho(x) \subset B_r(y)$. 
As a result, for each $X =(x, t) \in C_{r,\mu}(Y)$ and $\rho>0$ such that $B_{\rho }(x) \subset B_{r}(y)$, it follows from formula \eqref{cylinder measure formula} that
\begin{align} \notag 
\mu(C_{\rho, \mu}(X)) & = \sigma_n^{-1/n} \rho\mu(B_\rho(x))^{(n+1)/n} \\ \notag
& \geq \sigma_n^{-1/n} \rho K_0^{-(n+1)}  \left(\frac{\rho}{r}\right)^{(n+1)^2}  \mu(B_r(y))^{(n+1)/n} \\ \notag
& =K_0^{-(n+1)}  \left(\frac{\rho}{r}\right)^{(n+1)^2 +1} \sigma_n^{-1/n}r \mu(B_r(y))^{(n+1)/n} \\ \label{doubling-Cylin}
& = K_0^{-(n+1)}  \left(\frac{\rho}{r}\right)^{c_0} \mu(C_{r, \mu}(Y)),
\end{align}
where $c_0 = (n+1)^2 + 1 >0$ and $\sigma_n=\sigma_n(n)>0$.

\smallskip
Next, let us define
\begin{align*}
d(X) =
\begin{cases} 
\sup \left\{ \rho > 0: C_{\rho, \mu}(X) \subset C_{r, \mu}(Y) \right\} & \text{for } X\in C_{r, \mu}(Y)\cup \big(B_r(y)\times \{s\}\big), \\
0 & \text{for } X\in \partial_p C_{r, \mu}(Y)\cup \big(\partial B_{r}(y)\times \{s\}\big).
\end{cases}
\end{align*}
Also, let
\[
 M =  \sup_{X \in C_{r, \mu}(Y)} d(X)^{c_0} u(X).
\]
Note that we only need to consider the case $M>0$, as the assertion of the corollary is trivial if $M=0$. From the definition of $d$ above, it is not hard to see $d^{c_0} u \in C(\overline{C_{r, \mu}(Y)})$, and that there exists $X_0 = (x_0, t_0) \in C_{r,\mu}(Y)\cup \big(B_r(y)\times \{s\}\big)$ such that
\[
M = (2 r_0)^{c_0}  u(X_0) >0,\quad \text{where} \quad r_0 = \frac{1}{2} d(X_0)>0.
\]
It follows from the definition of $d(X)$ and Lemma \ref{cylinder-inclusion} that  for any $X \in C_{r_0, \mu}(X_0)$, 
\begin{align} \label{double inclusion}
C_{r_0, \mu}(X) \subset C_{2r_0, \mu}(X_0) \subset C_{r, \mu}(Y).
\end{align}
This implies that $d(X) \geq r_0$ for all $X\in C_{r_0,\mu}(X_0)$ and then
\[
2^{c_0}   u(X_0) = r_0^{-c_0} M \geq r_0^{-c_0} \sup_{C_{r_0, \mu}(X_0)} d^{c_0} u \geq \sup_{C_{r_0, \mu}(X_0)} u.
\]
From this, and by setting
\[
v(X) =  u(X) - \frac{1}{2} u(X_0)  \quad \text{for}\quad X \in C_{r, \mu}(Y),
\]
we infer that 
\begin{align} \label{growth-1-con}
v(X_0) = 2^{-1}  u(X_0) \geq   2^{-c_0-1} \sup_{C_{r_0, \mu}(X_0)} u 
>   2^{-c_0-1} \sup_{C_{r_0, \mu}(X_0)} v.
\end{align}

\smallskip
Let $\overline{\epsilon}_0=\overline{\epsilon}_0(n,\nu, K_0)>0$ and $\delta_0=\delta_0(n, \nu, K_0)>0$ be sufficiently small so that 
\begin{equation} \label{delta-1-def}
NK_0\overline{\epsilon}_0\leq 2^{-c_0-2} \quad \text{and}\quad N K_0\delta_0^{\frac{1}{n+1}} \leq 2^{-c_0-2},
\end{equation}
where $N = N(n,\nu)>0$ is the constant defined in the claim \eqref{claim-1} in the proof of Proposition \ref{growth-1}. 

\smallskip
Now,  we prove \eqref{bdd esti} with this choice of $\overline{\epsilon}_0$ in \eqref{delta-1-def} and the assumption \eqref{assum-cor-grow-1}.  From the choices of $\delta_0$ and $\overline{\epsilon}_0$,  \eqref{assum-cor-grow-1}, the claim \eqref{claim-1}, and \eqref{growth-1-con}, we infer that
\begin{equation} \label{Q-0-below}
\mu(\Omega) > \delta_0 \mu(C_{r_0, \mu}(X_0)), \quad \text{where} \quad \Omega = \{v>0\} \cap C_{r_0, \mu}(X_0).
\end{equation}
As $d(Y) =r$, it follows that
\[
u(Y) = r^{-c_0}u (Y) d(Y)^{c_0} \leq r^{-c_0}M.
\]
Moreover, from the definition of $\Omega$, we see that
\[ 
\Omega= \big\{ X: 2u (X) >  u(X_0) \big\} \cap C_{r_0, \mu}(X_0).
\] 
Therefore,
\begin{align*}
u^{+}(Y) & \leq r^{-c_0}M  = r^{-c_0} (2 r_0)^{c_0} u(X_0) \\
& = \left(\frac{2r_0}{r}\right)^{c_0 } \frac{1}{\mu(\Omega)} \int_{\Omega} u(X_0)\mu(x)\, dxdt \\
& \leq \left(\frac{2r_0}{r}\right)^{c_0} \frac{2}{\mu(\Omega)} \int_{\Omega} u^+(X)\mu(x)\, dxdt.
\end{align*}
By \eqref{doubling-Cylin}, \eqref{double inclusion}, and \eqref{Q-0-below}, 
it follows that 
\begin{equation*}%\label{case p=1}
\begin{aligned}
u^{+}(Y) & \leq   \frac{2^{c_0+1}\delta_0^{-1}}{\left(\frac{r}{r_0}\right)^{c_0} \mu(C_{r_0, \mu}(X_0))} \int_{C_{r_0, \mu}(X_0)} u^+(X)\mu(x) \, dxdt \\
 &\leq \frac{K_0^{n+1} 2^{c_0+1}\delta_0^{-1}}{\mu(C_{r, \mu}(Y))} \int_{C_{r, \mu}(Y)} u^+(X)\mu(x) \, dxdt \\
 & \leq \frac{N}{\mu(C_{r, \mu}(Y))} \int_{C_{r, \mu}(Y)} u^+(X)\mu(x) \,  dxdt,
\end{aligned}
\end{equation*}
where $N =N(n,\nu, K_0)=K_0^{n+1} 2^{c_0+1}\delta_0^{-1}>0$ with $c_0=(n+1)^2+1$ and $\delta_0$ defined in \eqref{delta-1-def}. Hence,  \eqref{bdd esti} for $q=1$ is proved and the proof is completed.
\end{proof}
\subsection{Second growth lemma}

The main results of this subsection are Proposition \ref{second-growth-lemma} and Corollary \ref{prop-up-lemma} below. To prove these results, we need a growth lemma in slant cylinders (Lemma \ref{slant-cylin-lemma} below).  For this purpose, let us introduce several notations. For $r>0$, and  $Y = (y,s) \in \mathbb{R}^{n+1}$ with $s>0$, the parabolic slant cylinder  $V_r(Y)$ is defined as
\begin{equation}\label{def-slant}
V_r(Y) =\Big\{(x, t) \in \mathbb{R}^{n+1}: \big| x- \frac{t}{s} y\big| < r, \ 0 < t< s \Big\}.
\end{equation}
We assume that there is $K \geq 1$ such that
\begin{equation}  \label{slant-cyl}
K^{-1}  r (\mu)_{B_R(y_0)}^{1/n} |y| \leq s \leq K r^2 (\mu)_{B_R(y_0)}^{1/n},
\end{equation}
where $y_0 \in \mathbb{R}^n$ and some $R \in [r, 2 K^2 r]$ such that
\begin{equation} \label{Vr-in}
x \in B_R(y_0) \text{ whenever there is some }  t \text{ such that } (x, t) \in V_r(Y).
\end{equation}
Observe that we can simply take $y_0= y/2$ and $R= (K^2/2 +1)r$, then $R \in [r, 2K^2 r]$ as $K \geq 1$, and \eqref{Vr-in} follows.  In other words,  \eqref{Vr-in} is not a condition, but an introduction to $y_0$ and $R$.
%with
%\begin{equation*} %\label{Ryr-def}
%  R \in [r, \hat{N}_0r] \quad \text{for} \quad \hat{N}_0 =  K^2/2 +1. 
%\end{equation*}

\smallskip
For a given slant cylinder $V_r(Y)$, we also write its parabolic boundary as
\begin{equation*} %\label{Slant-bdr}
\partial_p V_r(Y) = \big(\overline{B_r} \times \{0\}\big) \cup S,
\end{equation*}
where $S = \big\{(x, t): |x-(t/s)y|=r, \ 0 < t < s\big\}$. We now introduce the following growth lemma on slant cylinders, which is the key result in the subsection.
\begin{lemma}[Growth Lemma on Slant Cylinder] \label{slant-cylin-lemma} For $\nu\in(0,1)$, $K_0\in[1, \infty)$, and $K\in[1,\infty)$, there exist $\epsilon_1 =\epsilon_1 (n, \nu, K_0, K) \in (0,1)$  and $\beta_1 =\beta_1(n, \nu, K_0, K) \in (0,1)$ such that the following statement holds. Assume that \eqref{elliptic} holds in $V_r(Y)$, and  \eqref{slant-cyl} holds with some $r >0$ and $y_0, R$ as in \eqref{Vr-in}. Also assume that $\omega \in A_{1+\frac{1}{n}}$ and it satisfies 
\[ [\omega]_{A_{1+\frac{1}{n}}}\leq K_0 \quad \text{and} \quad  
 [[\omega]]_{B_R(y_0), \omega} \leq \epsilon_1.
 \] 
Then, for every function $u \in C^{2,1}(V_r(Y)) \cap C(\overline{V_r(Y)})$ satisfying
\[
\mathcal{L} u \leq 0 \quad \text{in} \quad V_r(Y) \quad \text{and} \quad   u (\cdot, 0) \leq 0 \quad \text{on} \quad B_r,
\]
it holds that
\[
 u(\cdot, s) \leq \beta_1 \sup_{\partial_p V_r(Y)}u^+ \quad \text{on} \quad B_{r/2}(y).
\]
\end{lemma}
\begin{proof}  As in the proof of Proposition \ref{growth-1}, the main idea is to freeze $\omega$ in $\mathcal{L}$, and then apply the ABP estimates (Theorem \ref{ABP theorem}) to a suitable barrier function.  Note that by Corollary~\ref{maximum principle}, we only need to prove the assertion of the lemma  when
\begin{equation} \label{u-normalized-1}
M = \sup_{\partial_p V_r(Y)} u >0.
\end{equation}
Let us denote
\[
w(x, t) = r^2 - |x - t \ell|^2, \quad \text{where} \quad \ell = \frac{1}{s} y,
\]
and let
\[ 
v(x, t) = e^{-\lambda t} w^2(x, t) \quad\text{for}\quad (x, t) \in \overline{V_r(Y)},
\]
where $\lambda>0$ is to be determined in \eqref{gamma-def} below. We note that
\[
D_iw = -2(x_i-t\ell_i) \quad \text{and} \quad D_{ij} w = -2\delta_{ij}.
\]
Moreover,
\[
v_t(x, t) = \left\{-\lambda w^2 +  4 \ell \cdot (x-t\ell) w \right\}e^{-\lambda t},
\]
and
\[
D_i v  =2 e^{-\lambda t} w D_i w, \quad D_{ij} v =2 e^{-\lambda t}\big(D_i w D_j w + w D_{ij}w \big).
\]
Therefore,
\begin{align*} 
\mathcal {L}v&=v_t - \omega(x) a_{ij} D_{ij}v\\  
&= e^{-\lambda t}\big[-\lambda w^2 +  4 \ell \cdot (x-t\ell) w  - 2\omega \big(a_{ij}D_iw D_j w + w a_{ij}D_{ij}w \big)\big] \\ 
& \leq  e^{-\lambda t}\big[-\lambda w^2 +  4 \ell \cdot (x-t\ell) w -2 \omega \nu|Dw|^2 + 4\omega \text{tr}(a_{ij})w\big],
\end{align*}
where in the last step we used \eqref{elliptic}, and $\text{tr}(\cdot)$ denotes the trace operator acting on matrices. As $|Dw|^2 = 4|x-t\ell|^2 = 4(r^2-w)$, we have
%\begin{equation*} 
\begin{align} \notag
 \mathcal {L}v & \leq e^{-\lambda t}\left\{-\lambda w^2 +  \big[ 4 \ell \cdot (x-t\ell) + 4\omega \text{tr}(a_{ij}) + 8\omega \nu \big]w -8 \nu  r^2   \omega \right\} \\ \notag
 & = e^{-\lambda t}\big\{-\lambda w^2 +  \big[ 4 \ell \cdot (x-t\ell) + 4(\omega)_{B_R(y_0)}  \text{tr}(a_{ij}) + 8(\omega)_{B_R(y_0)} \nu \big]w \\ \label{perturb-1-0921}
& \qquad -8 \nu  r^2  (\omega)_{B_R(y_0)}\big\}  + g(x, t),
\end{align}
%\end{equation*}
where $R>0$ and $y_0$ are defined in \eqref{Vr-in} and
\begin{align*}
g(x,t) & = e^{-\lambda t}\big\{ \big[ 4\big(\omega - (\omega)_{B_R(y_0)}  \big) \text{tr}(a_{ij}) + 8\big( \omega - (\omega)_{B_R(y_0)} \big) \nu \big] w \\
& \qquad  -8 \nu  r^2\big(\omega-  (\omega)_{B_R(y_0)}\big) \big\}.
\end{align*}
Due to \eqref{elliptic} and the fact that $|w| \leq r^2$, it follows that there is $N = N(n, \nu)>0$ such that
\begin{equation} \label{est-g-0508}
|g(x,t)| = N r^2e^{-\lambda t}\big| \omega(x) - (\omega)_{B_R(y_0)} \big| \quad \text{on} \quad V_r(Y).
\end{equation}
Also, by \eqref{slant-cyl} and \eqref{A-1-1/n}, we have
\[
| \ell \cdot (x-t\ell)| \leq |\ell| |x-t \ell| \leq K (\mu)_{B_R(y_0)}^{-1/n}  \leq K  (\omega)_{B_R(y_0)}.
\]
From this, \eqref{elliptic}, and  $K\geq 1$, there exists $N_0 = N_0(n, \nu)>0$ such that
\[
\big|  4 \ell \cdot (x-t\ell) + 4(\omega)_{B_R(y_0)}  \text{tr}(a_{ij}) + 8(\omega)_{B_R(y_0)} \nu \big| \leq N_0K (\omega)_{B_R(y_0)}.
\]
Therefore, we can infer from \eqref{perturb-1-0921} that
\begin{align*}
 \mathcal{L}v\leq e^{-\lambda t}\big[-\lambda w^2 +   N_0K(\omega)_{B_R(y_0)} w  - 8\nu r^2  (\omega)_{B_R(y_0)}\big] + g(x, t) \quad \text{in} \quad V_r(Y).
\end{align*}
By the Cauchy-Schwarz inequality, we have
\[
N_0K(\omega)_{B_R(y_0)} w\leq \frac{N_0^2K^2(\omega)_{B_R(y_0)}w^2}{32\nu r^2}+8\nu r^2(\omega)_{B_R(y_0)}.
\]
Then, taking 
\begin{equation} \label{gamma-def}
\lambda =  \frac{N_1K^2 (\omega)_{B_R(y_0)}}{r^2}>0 \quad \text{with}\quad N_1(n,\nu)=\frac{N_0^2}{16\nu},
\end{equation}
we see that
\[
-\lambda w^2 +  N_0 K (\omega)_{B_R(y_0)} w  - 8\nu r^2  (\omega)_{B_R(y_0)} \leq 0 \quad \text{in}\quad V_r(Y).
\]
Because of this, it follows that
\[ 
\mathcal{L} v = v_t - \omega(x) a_{ij} D_{ij}v \leq  g(x, t) \quad \text{in}\quad V_r(Y). \]
Next, we define
\[
\phi(x, t) = u(x, t) -M + \frac{M}{r^4} v(x,t) \quad \text{in}\quad \overline{V_{r}(Y)}.
\]
It then follows from the assumption on $u$ that
\[
\mathcal{L} \phi = \phi_t - \omega(x) a_{ij} D_{ij}\phi \leq \frac{M}{r^4}g \quad \text{in} \quad V_r(Y).
\]
We also note that
\begin{equation} \label{phi-para-boun}
\phi \leq 0 \quad \text{on} \quad \partial_p V_r(Y) = \big(\overline{B_r} \times \{0\}\big) \cup S,
\end{equation}
where $S = \{(x,t): |x-t\ell|=r, \ 0 < t < s\}$ in which $\ell=y/s$. Precisely, for all $x \in \overline{B_r}$, due to the assumption that $u(x,0) \leq 0$, we see that
\[
\phi(x,0) = u(x,0) - M + \frac{M}{r^4}v(x,0)  \leq M \Big[ \frac{(r^2 -|x|^2)^2}{r^4}  - 1\Big] \leq 0.
 \]
On the other hand, for all $(x, t) \in S$, as $v(x, t) =0$, it follows by the definition of $M$ in \eqref{u-normalized-1} that
\[
\phi (x, t) =u(x, t) - M \leq 0.
\]
Hence \eqref{phi-para-boun} is verified. 

\smallskip
Now, observe that from \eqref{Vr-in}, and the definition of $V_r(Y)$, 
it follows that
\begin{equation} \label{rR-BV}
 V_r(Y) \subset B_{R}(y_0) \times (0, s).
\end{equation}
Using \eqref{phi-para-boun} and \eqref{rR-BV}, we can apply the ABP estimates (see Theorem \ref{ABP theorem}) for $\phi$ in $V_r(Y)$ to obtain
\[ 
\sup_{V_r(Y)} \phi^+  \leq N M r^{-4}  R^{\frac{n}{n+1}} \|g\|_{L^{n+1} (V_r(Y), \mu)} 
\]
for $N = N(n, \nu)>0$. Then, using \eqref{Vr-in},  \eqref{est-g-0508}, and \eqref{gamma-def}, we obtain
\begin{align*}
& \sup_{V_r(Y)} \phi^+ \\
& \leq N K^4M  R^{\frac{n}{n+1}-2} \Big(\int_{V_r(Y)} e^{-\lambda (n+1)t} \big|(\omega)_{B_R(y_0)} - \omega(x) \big|^{n+1} \omega^{-n}(x) dxdt \Big)^{1/(n+1)} \\
& \leq N K^4M R^{\frac{n}{n+1} -2} \Big(\frac{1}{\lambda}\int_{B_R(y_0)} \big|(\omega)_{B_R(y_0)} - \omega(x) \big|^{n+1} \omega^{-n}(x) dx \Big)^{1/(n+1)} \\
& \leq N_2 K^{4-\frac{2}{n+1}} M   \Big(\frac{1}{\omega(B_R(y_0))}\int_{B_R(y_0)} \big|(\omega)_{B_R(y_0)} - \omega(x) \big|^{n+1} \omega^{-n}(x) dx \Big)^{1/(n+1)} \\
& = N_2K^{4-\frac{2}{n+1}} M [[\omega]]_{B_R(y_0), \omega},
\end{align*}
where $N_2=N_2(n,\nu)>0$. From the definition of $v$, we note that
\[
r^{-4} v(x,s) \geq \frac{9}{16} e^{-\lambda s} \quad \text{for all}\quad x \in B_{r/2}(y).
\]
Also, it follows from the choice of $\lambda$ in  \eqref{gamma-def}, and \eqref{slant-cyl}  that
\[
\begin{split}
\lambda s & = r^{-2} N_1K^2 (\omega)_{B_R(y_0)} s \leq N_1K^3 (\omega)_{B_R(y_0)} (\mu)^{1/n}_{B_{R}(y_0)} \leq N_1K^3  [\omega]_{A_{1+\frac{1}{n}}} \\
& \leq N_1(n,\nu)K^3K_0.
\end{split}
\]
Therefore, combining the last three estimates,  and using the assumption $[[\omega]]_{B_R(y_0), \omega} \leq \epsilon_1$, 
we see that for all $x \in B_{r/2}(y)$,
\begin{align*}
u(x, s)  & \leq  \Big[ 1 - r^{-4}v(x, s) + N_2K^{4-\frac{2}{n+1}} \epsilon_1\Big]M  \\ 
&  \leq \Big[ 1 - \frac{9}{16}e^{-N_1K^3 K_0} + N_2K^{4-\frac{2}{n+1}}\epsilon_1\Big]M.
\end{align*}
Finally, by taking
\begin{equation}\label{def-beta_2}
\begin{split}
\epsilon_1 & =\epsilon_1(n,\nu, K_0, K) = \frac{9}{32N_2K^{4-\frac{2}{n+1}}}e^{-N_1K^3 K_0}\quad \text{and}\ \\
\beta_1 & =\beta_1(n,\nu,K_0, K)=1-\frac{9}{32}e^{-N_1 K^3 K_0},
\end{split}
\end{equation}
where $N_1=N_1(n,\nu)>0$ and $N_2=N_2(n,\nu)>0$, we see that $\beta_1 \in (0,1)$ and
\[
u(\cdot, s)  \leq \beta_1 M \quad \text{on} \quad B_{r/2}(y).
\]
The proof of the lemma is completed.
\end{proof}

%===
%\begin{remark} \label{monotone-beta-1} It follows from \eqref{def-beta_2} that $\epsilon_1(n,\nu, K_0, K)$ is decreasing in $K$ meanwhile $\beta_1(n,\nu,K_0, K)$ is increasing in $K$.
%\end{remark}
\smallskip
Now, for $Y=(y, s)\in \R^{n+1}$, $r>0$, and $h\in(0,1]$, we denote the following type of weighted parabolic cylinders with height $hr^2(\mu)_{B_r(y)}^{1/n}$ by
\[C^{h}_{r,\mu}(Y)=B_r(y)\times\big(s-hr^2(\mu)_{B_r(y)}^{1/n}, s\big).
\]
The following result, which is known as the second growth lemma, is the first main result of this subsection.
\begin{proposition}[Second Growth Lemma] \label{second-growth-lemma} For $\nu\in(0,1)$,  $r>0$, $\rho >0$,  $h\in(0,1]$, and $K_0\in[1,\infty)$, there exist $\epsilon_2 =\epsilon_2 (n, \nu, K_0, r/\rho, h) \in (0,1)$ sufficiently small and 
$\beta_2= \beta_2(n, \nu, K_0, r/\rho, h)\in (0,1)$ such that the following assertion holds.  Assume that \eqref{elliptic} holds in $C_{r,\mu}^h(Y)$, and that $\omega \in A_{1+\frac{1}{n}}$ satisfies
 \[ [\omega]_{A_{1+\frac{1}{n}}}\leq K_0 \quad \text{and} \quad  [[\omega]]_{\textup{BMO}(B_{2r}(y), \omega)} \leq \epsilon_2. 
 \]
Then, for every $u \in C^{2,1}(C^h_{r, \mu}(Y)) \cap C(\overline{C^h_{r,\mu}(Y)})$ satisfying  $\mathcal{L}u \leq 0$ in $C^h_{r,\mu}(Y)$ and
\[
u(\cdot, \tau) \leq 0 \quad \text{in} \quad   B_\rho(z),
\]
with some $z \in \mathbb{R}^n$, $\rho \in (0, r]$, and $\tau$ satisfying
\begin{equation}\label{tau-condition}
B_\rho(z) \subset B_r(y) \quad \text{and} \quad s -hr^2(\mu)_{B_r(y)}^{1/n} \leq \tau \leq s - h\Big(\frac{ r^2}{2} + \rho^2\Big)(\mu)_{B_r(y)}^{1/n},
\end{equation}
it holds that
\begin{equation} \label{second-lemma-conclusion}
\sup_{C^h_{r/2, \mu}(Y)} u^+ \leq \beta_2 \sup_{C^h_{r, \mu}(Y)} u^+.
\end{equation}
\end{proposition}
\begin{proof} 
By taking the linear translation $(x, t) \mapsto (x+z, t+\tau)$ and replacing $Y=(y, s)$ by $(y-z, s-\tau)$, we can assume without loss of generality that $(z,\tau)=(0,0)$. 
Then, in this setting, the assumption \eqref{tau-condition} is reduced to
\begin{equation} \label{r-rho-relation}
B_\rho(0) \subset B_r(y), \quad h\Big(\frac{ r^2}{2} + \rho^2\Big)(\mu)_{B_r(y)}^{1/n} \leq s \leq hr^2(\mu)_{B_r(y)}^{1/n},
\end{equation}
which particularly implies $\rho < r/\sqrt{2}$. 
Now, let $Y' = (y',s') \in C^h_{r/2, \mu}(Y)$ be fixed.  
Our goal is to apply Lemma \ref{slant-cylin-lemma} for the slant cylinder $V_\rho(Y')$ to derive the assertion \eqref{second-lemma-conclusion}. To this end, we need to verify the conditions in Lemma \ref{slant-cylin-lemma}. 

\smallskip
Firstly, we verify the condition \eqref{slant-cyl}. Observe that as $Y' = (y',s') \in C^h_{r/2, \mu}(Y)$, it follows that
\[
y' \in B_{r/2}(y) \quad \text{and} \quad s -\frac{hr^2}{4}(\mu)_{B_{r/2}(y)}^{1/n} < s' < s.
\]
From this, the bounds of $s$ in \eqref{r-rho-relation}, we infer that
\begin{equation} \label{s'-1}
h\rho^2(\mu)_{B_r(y)}^{1/n}  < s' <hr^2(\mu)_{B_r(y)}^{1/n} =h \Big(\frac{r}{\rho} \Big)^2 \rho^2(\mu)_{B_r(y)}^{1/n} . 
\end{equation}
In addition, as $0 \in B_r(y)$ due to \eqref{r-rho-relation}, it follows that $|y| = |y-0| < r$. From this and the triangle inequality, we obtain 
\[
|y'| \leq |y| + |y-y'| < \frac{3 r}{2},
\]
and note that $\rho\leq r$, consequently
\begin{equation} \label{s'-2}
 \Big( \frac{2h\rho^2 }{3 r^2}\Big)\rho(\mu)_{B_r(y)}^{1/n} |y'| <  h \rho^2(\mu)_{B_r(y)}^{1/n} < s'.
\end{equation}
Then, due to $h\in (0,1]$ and by taking 
\begin{equation}\label{K-formula}
K = \max\Big\{ \frac{3 r^2}{2h\rho^2}, h\Big(\frac{r}{\rho} \Big)^2  \Big\} =\frac{3}{2h}\Big(\frac{r}{\rho}\Big)^2\geq 1, 
\end{equation}
we infer from \eqref{s'-1} and \eqref{s'-2} that
\[
K^{-1} \rho(\mu)_{B_r(y)}^{1/n} |y'| \leq s' \leq K \rho^2(\mu)_{B_r(y)}^{1/n}.
\]
This implies that the condition \eqref{slant-cyl} for the slant cylinder $V_\rho(Y')$ holds. 

\smallskip
Next, it is clear that condition \eqref{Vr-in} with $\rho$ in place of $r$ and $r$ in place of $R$ also holds. To see this, note that $V_\rho(Y') \subset C^h_{r, \mu}(Y)$ as both the top and bottom bases of $V_\rho(Y')$ are in $C^h_{r, \mu}(Y)$. Then, for any $(x,t) \in V_\rho(Y')$, we have $x \in B_r(y)$. 

\smallskip
Finally, we verify that $r \in [\rho, 2K^2\rho]$. Note that as $r\geq \rho$ and $h\in (0,1)$, we have
\[
\rho\leq r<\frac{9r^3}{4h^2\rho^3}r=\frac{9r^4}{4h^2\rho^4}\rho=K^2\rho \leq 2K^2\rho.
\]

\smallskip
Now, from the choice of $K = K(r/\rho,h)$ in \eqref{K-formula}, let
\begin{equation}\label{epsilon-2, beta-2}
\begin{split}
 \epsilon_2(n,\nu, K_0, r/\rho, h) &= \epsilon_1 (n, \nu,K_0, K) \in (0,1)\quad \text{and}\\
\beta_2 (n,\nu, K_0, r/\rho, h)&= \beta_1(n, \nu, K_0, K) \in (0,1),
\end{split}
 \end{equation}
where the numbers $\epsilon_1$ and $\beta_1$ are defined in \eqref{def-beta_2} in  Lemma \ref{slant-cylin-lemma}. Then, we apply Lemma \ref{slant-cylin-lemma} to the slant cylinder $V_\rho(Y')$ to infer that
\[
u(Y') \leq \beta_2 \sup_{\partial_pV_\rho(Y')} u^+ \leq \beta_2 \sup_{C^h_{r,\mu}(Y)} u^+.
\]
As $Y'$ is arbitrary in $C^h_{r/2, \mu}(Y)$, we conclude that
\[
\sup_{C^h_{r/2, \mu}(Y)} u^+ \leq \beta_2 \sup_{C^h_{r,\mu}(Y)} u^+.
\]
The proof of the lemma is completed.
\end{proof} 
\begin{remark}\label{h-dependce-1}  If $n,\nu, K_0, r/\rho$ defined in Proposition \ref{second-growth-lemma} are fixed, then the function $h \mapsto \epsilon_2(n,\nu, K_0, r/\rho,h)$ is increasing in $(0,1]$, and the function $h \mapsto \beta_2(n,\nu, K_0, r/\rho, h)$ is decreasing in $ (0,1]$.
\end{remark}
\begin{proof} It follows from \eqref{def-beta_2}, \eqref{K-formula}, and \eqref{epsilon-2, beta-2} that there are positive constants $c_1, c_2$ depending on $n, \nu, K_0, r/\rho$ so that
\begin{align*}
\epsilon_2(n,\nu, K_0, r/\rho,h) & = c_1 h^{4-\frac{2}{n+1}} e^{- c_2 h^{-3}} \quad \text{and} \\
\beta_2(n,\nu, K_0, r/\rho, h) & = 1-\frac{9}{32}e^{- c_2 h^{-3}}.
\end{align*}
The assertions then follow. %We find $\epsilon_1$ is decreasing and $\beta_1$ is increasing as functions of  $K\in[1,\infty)$. On the other hand, from the choice of $K$ in \eqref{K-formula}, we have $h \mapsto K$ is decreasing in $h\in(0,1]$. 
\end{proof} 

%====
We note that the numbers $\epsilon_2$ and $\beta_2$ defined in Proposition \ref{second-growth-lemma} both depend on the ratio $r/\rho$. To conclude this subsection, we introduce the following corollary, sometimes referred to as the prop-up lemma, which removes these mentioned dependences. This corollary is the second main result in this subsection and it plays a key role in the proof of Theorem \ref{Harnack inequality}.
%====
\begin{corollary}[Prop-up Lemma]\label{prop-up-lemma} For every $h \in (0,1]$ and $K_0\in[1,\infty)$, there exist $\gamma = \gamma(n, \nu, K_0, h)>0$, and sufficiently small   $\bar{\epsilon}_2= \bar{\epsilon}_2(n, \nu, K_0, h)\in (0,1)$ such that the following assertion holds. Assume that \eqref{elliptic} holds in $C_{r,\mu}(Y)$ for some $r>0$ and $Y = (y,s) \in \mathbb{R}^{n+1}$, and assume that $\omega \in A_{1+\frac{1}{n}}$ satisfies
\begin{equation} \label{bar-epsilon-cond}
 [\omega]_{A_{1+\frac{1}{n}}}\leq K_0 \quad \text{and} \quad  [[\omega]]_{\textup{BMO}(B_{2r}(y), \omega)} \leq \bar{\epsilon}_2.
\end{equation}
Then, for every $u \in C^{2,1}(C_{r, \mu}(Y)) \cap C(\overline{{C}_{r,\mu}(Y)})$ satisfying
\[
u \geq 0 \quad \text{and} \quad  \mathcal{L} u \geq 0 \quad \text{in} \quad C_{r, \mu}(Y).
\]
it follows that 
\begin{equation} \label{prop-up-assert}
 \inf_{B_\rho(z)} u(\cdot, \tau)\leq  \Big(\frac{4r}{\rho}\Big)^{\gamma} \inf_{B_{r/2}(y)} u(\cdot, \sigma),
\end{equation}
for any $\rho \in (0, r]$, $z \in \mathbb{R}^{n+1}$, and $\tau, \sigma \in \mathbb{R}$ such that
\begin{equation} \label{prop-cond-ball}
B_\rho(z) \subset B_r(y) \quad \text{and} \quad s -r^2 (\mu)_{B_r(y)}^{1/n} \leq \tau < \tau + h r^2 (\mu)_{B_r(y)}^{1/n} \leq \sigma \leq s.
\end{equation}
\end{corollary}
\begin{proof} 
We follow the approach used in \cite{FeSa}. By taking a linear translation, we can assume with loss of generality that $(z,\tau) =(0,0)$. In this setting, \eqref{prop-cond-ball} is reduced to  
\begin{equation*} %\label{reduced-gap}
B_\rho \subset B_r(y) \quad \text{and} \quad s -r^2 (\mu)_{B_r(y)}^{1/n} \leq 0 <  h r^2 (\mu)_{B_r(y)}^{1/n} \leq\sigma \leq s.
\end{equation*}
Let us denote $\sigma_0 = h r^2 (\mu)_{B_r(y)}^{1/n}$, and we begin with the claim that there is $\gamma_1 = \gamma_1(n, \nu, K_0, h)>0$ such that
\begin{equation} \label{prop-up-assert-reduced}
\inf_{B_\rho} u(\cdot, 0)\leq\Big(\frac{4r}{\rho}\Big)^{\gamma_1} \inf_{B_{r/2}(y)} u(\cdot, \sigma_0).
\end{equation}

%\smallskip
To prove the claim, let us denote
\[
m=\displaystyle{\inf_{B_{\rho}} u(\cdot, 0)}.
\]
As \eqref{prop-up-assert-reduced} is trivial when $m=0$, from now on, we assume that $m>0$. The main idea is to build a stair with blocks of cylinders of certain heights to climb from the base at $\{t= \tau\}$ to the top level at $\{t= \sigma_0\}$, and apply Proposition~\ref{second-growth-lemma} to each block iteratively to achieve \eqref{prop-up-assert-reduced}. To this end, for every $k=0, 1, 2,\ldots$, let
\[ 
r_k = 4^{-k} r \quad \text{and} \quad y^k =y^{\infty}+4^{-k}(y-y^{\infty}),
\]
where 
\begin{equation*}  
y^{\infty}=\Big\{
 \begin{array}{ll}
\rho y/(\rho-r) &  \quad \text{if}\quad \rho\in(0,r), \\
 0&  \quad \text{if}\quad \rho=r.
\end{array} \Big.
\end{equation*} 
Also, for $k =0,1, 2,\ldots$, let us write
\begin{equation}\label{definition of cylinders}
 B^k =B_{r_k}(y^k),\quad C^k=C^h_{r_k,\mu}(Y^k)=B^k\times \big(0, s_k\big), 
 \end{equation}
where $s_k =hr_k^2(\mu)^{1/n}_{B^k}$ and $Y^k=(y^k, s_k)$.
Next, as $\rho \in (0, r]$, there exists a unique $k_0\in\N$ such that 
\begin{equation}\label{rho}
4^{-k_0-1}r<\rho\leq 4^{-k_0}r.
\end{equation}
Due to \ $y^{\infty}+\frac{\rho}{r}(y-y^{\infty})=0$, and the constructions of $B^k$, it is not hard to see that
\begin{equation*} %\label{balls inclusion}
\begin{aligned}
& B^{k_0+1}\subset B_{\rho}\subset B^{k_0}\subset \cdots \subset B^1\subset B^0=B_r(y) \qquad \text{and}\\
& C^{k_0}\subset C^{k_0-1}\subset \cdots \subset C^1 \subset C^0 =C^h_{r,\mu}(Y^0).
\end{aligned}
\end{equation*}

\smallskip
Now, let 
\begin{equation}\label{bar-epsilon}
\bar{\epsilon}_{2}=\epsilon_2(n,\nu, K_0, 8, h) \in (0,1)\quad \text{and}\quad \bar{\beta}=\beta_2(n,\nu, K_0, 8, h)\in (0,1),
\end{equation}
where $\epsilon_2, \beta_2$ are the numbers defined in Proposition \ref{second-growth-lemma}. Also let
 \[ 
 m_k = \inf_{B_{r_k/2}(y^k)} u (\cdot, s_k), \quad k =0, 1,\ldots, k_0.
\] 
Our aim is to apply Proposition \ref{second-growth-lemma} on $C^k$ with $k \in \{0, 1, \ldots, k_0\}$ to derive  \eqref{prop-up-assert-reduced}.  We claim that 
\begin{equation} \label{claim-prop-up-as}
m_k  \geq (1-\bar\beta)^{k_0+1 -k} m, \quad \forall \ k = 0, 1, \ldots, k_0.
\end{equation}
We split the proof of the claim into 3 steps starting with $k=k_0$ and iteratively moving down to $k=0$.

\smallskip

\smallskip
\noindent
\textbf{Step I}.  We prove \eqref{claim-prop-up-as} when $k =k_0$. We note that as $B^{k_0+1}\subset B_{\rho}$, it holds that
 \[
 \inf_{B^{k_0+1}} u(\cdot, 0)\geq m>0.
 \]
Then, set $v_{k_0}=1-m^{-1}u$. Since $u\geq 0$ in $C_{r,\mu}(Y)$, we see that
\[ v_{k_0} \leq 1 \quad \text{and} \quad \mathcal{L} v_{k_0}\leq 0 \quad \text{in} \quad C^{k_0}.
\]
Moreover,
\[
v_{k_0}(\cdot,0)\leq 0 \quad \text{in} \quad  B^{k_0+1}.
\]
Recall the definition of $C^{k_0}$ in \eqref{definition of cylinders},  and set
\begin{align*}
z^{k_0}=y^{k_0+1}, \quad \rho_{k_0}=2^{-2k_0-3}r,\quad \tau_{k_0}=\tau=0.
\end{align*}
Then, note that
\[
 r_{k_0}/\rho_{k_0}=8, \quad B_{\rho_{k_0}}(z^{k_0})\subset B^{k_0}\quad \text{and}\quad  0=\tau_{k_0}\leq \big(r_{k_0}^2/2-\rho_{k_0}^2\big)h(\mu)^{1/n}_{B^{k_0}}.
\]
We apply Proposition \ref{second-growth-lemma} to $v_{k_0}$ on $C^{k_0}$, in which $Y^{k_0}$, $r_{k_0}$, $z^{k_0}$, $\rho_{k_0}$ and $\tau_{k_0}$ are in place of $Y$, $r$, $z$, $\rho$ and $\tau$, respectively.  We then infer that
\[
 \sup_{B_{r_{k_0}/2}(y^{k_0})}v_{k_0}(\cdot, s_{k_0})\leq \bar{\beta} \sup_{C^{k_0}}v_{k_0}\leq \bar{\beta}.
 \]
From this and the definition of $v_{k_0}$, we conclude that 
\begin{equation}\label{first-step-conclusion}
m_{k_0}=\inf_{B_{r_{k_0}/2}(y^{k_0})}u(\cdot, s_{k_0})\geq(1-\bar \beta)m.
\end{equation}
Hence, \eqref{claim-prop-up-as} holds when $k = k_0$.

\smallskip
\noindent
\textbf{Step II}.  We verify \eqref{claim-prop-up-as} when $k = k_0 -1$. From \eqref{first-step-conclusion} and as $m>0$, we see that $m_{k_0} >0$. For $v_{k_0-1}=1-(m_{k_0})^{-1}u\leq 1$, we see that
\[ 
\mathcal{L} v_{k_0-1}\leq 0 \quad \text{in}  \quad C^{k_0-1}, \]
and 
\[ 
v_{k_0-1}(\cdot,s_{k_0})\leq 0 \quad \text{in} \quad  B_{r_{k_0}/2}(y^{k_0}).
\]
Similarly as in {\bf Step I}, we set
\begin{equation}\label{prop-up iteration}
\begin{aligned}
z^{k_0-1}=y^{k_0},\quad \rho_{k_0-1}=2^{-2(k_0-1)-3}r,\quad \tau_{k_0-1}=s_{k_0}.
\end{aligned}
\end{equation}
It follows by direct computation that
\begin{align*}
 B_{\rho_{k_0-1}}(z^{k_0-1})\subset B^{k_0-1} \quad \text{and}\quad
 0<s_{k_0}=\tau_{k_0-1}\leq \big(r_{k_0-1}^2/2-\rho_{k_0-1}^2\big)h(\mu)^{1/n}_{B^{k_0-1}}.
\end{align*}
Then, as in {\bf Step I}, we apply Proposition \ref{second-growth-lemma} to $v_{k_0-1}$ on $C^{k_0-1}$ to get
\[
\inf_{B_{r_{k_0-1}/2}(y^{k_0-1})}u(\cdot, s_{k_0-1})\geq(1-\bar \beta)m_{k_0}.
\]
From this and  \eqref{first-step-conclusion}, we infer that
\begin{equation}\label{second-step-conclusion}
m_{k_0-1}  \geq(1-\bar \beta)^2m.
\end{equation}
Hence \eqref{claim-prop-up-as} holds when $k = k_0-1$.

\smallskip
\noindent
\textbf{Step III}. We carry out the iteration to prove \eqref{claim-prop-up-as} for $k\leq k_0-2$. For $l=k_0-2, k_0-3, \cdots, 1, 0$, we set
\[
v_{l}=1-(m_{l+1})^{-1}u,
\]
and as in \eqref{prop-up iteration}, we also set
\begin{equation*}
\begin{aligned}
z^{l}=y^{l+1}, \quad \rho_{l}=2^{-2l-3}r,\quad \tau_{l}=s_{l+1}.
\end{aligned}
\end{equation*}
As in {\bf Step II}, we apply Proposition \ref{second-growth-lemma} inductively to $v_l$ in $C^l$ for $l=k_0-2, k_0-3, \ldots, 1, 0$. We obtain
\begin{equation*}%\label{third-step-conclusion}
m_{l}\geq(1-\bar \beta)m_{l+1}, \quad \text{for all}\quad l=k_0-2, k_0-3,\ldots, 1, 0.
 \end{equation*}
 This, \eqref{first-step-conclusion}, and \eqref{second-step-conclusion} imply \eqref{claim-prop-up-as}.  Therefore, the claim \eqref{claim-prop-up-as} is proved.
 
\smallskip
Now, we use \eqref{claim-prop-up-as} to derive \eqref{prop-up-assert-reduced}.  Note that \eqref{claim-prop-up-as} when $k=0$ can be rewritten as
\begin{equation}\label{conclusion-1}
 \inf_{B_{\rho}} u(\cdot,0)\leq(1-\bar \beta)^{-k_0-1}\inf_{B_{r/2}(y)} u(\cdot, \sigma_0).
\end{equation}
Then, by setting 
\begin{equation}\label{gamma_1 choice}
\gamma_1 =\gamma_1 (n,\nu,K_0,h) =-\log_4\big(1-\bar\beta\big)>0,
\end{equation}
and using \eqref{rho}, we infer that
\[
(1-\bar\beta)^{-k_0-1}\leq \Big(\frac{4r}{\rho}\Big)^{\gamma_1}.
\]
From this and \eqref{conclusion-1}, it follows that
\begin{equation} \label{prop-up-06-11}
 \inf_{B_\rho} u(\cdot, 0)\leq\Big(\frac{4r}{\rho}\Big)^{\gamma_1} \inf_{B_{r/2}(y)} u(\cdot, \sigma_0).
\end{equation}
Hence \eqref{prop-up-assert-reduced} is proved. This also implies \eqref{prop-up-assert} when $\sigma = \sigma_0$.

\smallskip
It remains to prove \eqref{prop-up-assert} for all $\sigma \in [\sigma_0, s]$. We note that with the fixed numbers $n,\nu, K_0$, it follows from \eqref{bar-epsilon}, and Remark \ref{h-dependce-1} that the map $h \mapsto \bar \epsilon_2(n,\nu, K_0,h)$ is increasing, and the map $h \mapsto  \bar \beta(n,\nu, K_0, h)$ is decreasing on $(0,1]$. Then, from the choice of $\gamma_1$ in \eqref{gamma_1 choice}, we also infer that the map $\hat{\gamma}: (0, 1] \mapsto \R$ defined by
\[ \hat{\gamma}(h)= \gamma_1(n,\nu, K_0, h), \quad h \in (0, 1]\]
is decreasing in  $(0,1]$. 

\smallskip
Now, for $\sigma\in [\sigma_0, s]$, let $\bar{h} \in [h, 1]$ such that $\sigma=\bar{h}r^2(\mu)_{B_r(y)}^{1/n}$. Then, under the assumption \eqref{bar-epsilon-cond} and as $\bar{\epsilon}_2(n,\nu, K_0,h) \leq \bar{\epsilon}_2(n,\nu, K_0, \bar{h})$, we can apply \eqref{prop-up-assert-reduced}, with $\bar{h}$ in place of $h$, to obtain
\[
\inf_{B_\rho(z)} u(\cdot, 0)\leq\Big(\frac{4r}{\rho}\Big)^{\hat{\gamma}(\bar h)} \inf_{B_{r/2}(y)} u(\cdot, \sigma).
\]
From this, the fact that $\hat{\gamma}(\bar{h}) \leq \hat{\gamma}(h)$, and as $\rho \in (0, r]$, we have
\[
\inf_{B_\rho(z)} u(\cdot, 0) \leq\Big(\frac{4r}{\rho}\Big)^{\gamma} \inf_{B_{r/2}(y)} u(\cdot, \sigma),
\]
where $\gamma = \hat{ \gamma}(h) = \gamma_1(n,\nu, K_0, h)$. This implies \eqref{prop-up-assert}, and the proof of the lemma is completed.
\end{proof}

\subsection{Third growth lemma} 
The main result of this subsection is Proposition \ref{third-growth}, referred to as the third growth lemma. To prove this proposition, we need several lemmas. Let us recall the type of cylinders defined in \eqref{hat cylinder} 
\[
\hat{C} = \hat{C}_{\rho, \mu}(X_0)=B_{\rho}(x_0)\times \big(t_0+\rho^2(\mu)_{B_\rho(x_0)}^{1/n}, t_0+K_1\rho^2(\mu)_{B_\rho(x_0)}^{1/n}\big),
\]
where $X_0=(x_0,t_0)\in \R^{n+1}$, $\rho>0$, and $K_1\in (1,\infty)$ is fixed.  Corresponding to $\hat{C}$, we also denote 
\begin{equation}\label{cylinder U}
U = U_{\rho, \mu}(X_0)=B_{2\rho}(x_0)\times \big(t_0, t_0+K_1\rho^2(\mu)_{B_\rho (x_0)}^{1/n}\big).
\end{equation}
We note that $\hat{C} \subset U$. The following type of growth lemma is needed in the proof of Lemma \ref{alternative-lemma}.
\begin{lemma} \label{up-lemma-5} For $\nu\in(0,1)$, $K_0\in[1, \infty)$, $K_1\in (1,\infty)$,  there exist 
\[ \hat{\epsilon}_2 =\hat{\epsilon}_2 (n, \nu, K_0, K_1) \in (0,1) \quad \text{and} \quad \hat{\beta}_2 =\hat{\beta}_2(n, \nu, K_0, K_1) \in (0,1)
\] 
such that the following statement holds. Assume that \eqref{elliptic} holds in $U$, and that $\omega \in A_{1+\frac{1}{n}}$ satisfies 
\[ [\omega]_{A_{1+\frac{1}{n}}}\leq K_0 \quad \text{and} \quad  
 [[\omega]]_{B_{2\rho}(x_0), \omega} \leq \hat{\epsilon}_2,
 \] 
where $U$ is defined in \eqref{cylinder U} with some $\rho>0$ and $X_0 = (x_0, t_0) \in \mathbb{R}^n \times \mathbb{R}$.  Then, for every function $u \in C^{2,1}(U) \cap C(\overline{U})$ satisfying
\[
\mathcal{L} u \leq 0 \quad \text{in} \quad U \quad \text{and} \quad   u (\cdot, t_0) \leq 0 \quad \text{on} \quad B_{\rho/2}(x_0),
\]
it holds that
\[
\sup_{\hat{C}} u^+ \leq \hat{\beta}_2 \sup_{U}u^+.
\]
\end{lemma}
%====
\begin{proof}
Applying a linear translation, we may assume $X_0=(x_0,t_0)=(0,0)$, then we write $B_\rho=B_\rho(0)$. For any 
\[
Y=(y,s)\in B_{\rho/2}\times \big(\rho^2(\mu)_{B_\rho}^{1/n}, K_1\rho^2(\mu)_{B_\rho}^{1/n}\big),
\]
 we define the slant cylinder $V_{\rho}(Y)$ as in \eqref{def-slant}. Due to $K_1>1$ and $|y|< \rho/2$, it follows that
\begin{equation}\label{condition K}
K_1^{-1}\rho(\mu)_{B_\rho}^{1/n}|y|\leq s\leq K_1\rho^2(\mu)_{B_\rho}^{1/n}.
\end{equation}
Note that $V_{\rho}(Y) \subset U$, so \eqref{Vr-in} holds with $R=2\rho$ and $y_0=0$. Next, we verify the condition in \eqref{slant-cyl} with $R=2\rho$ and $y_0=x_0 =0$. By the doubling property of $\mu$ in \eqref{doubling property}, we have 
\[
\mu(B_{2\rho}) \leq 2^{n(n+1)}K_0^n \mu(B_{\rho}).
\]
Therefore, it follows from \eqref{cylinder measure formula} that
\[
\rho(\mu)_{B_{\rho}}^{1/n}=\sigma_n^{-1/n}\mu(B_{\rho})^{1/n}\geq 2^{-(n+1)}K_0^{-1}\sigma_n^{-1/n}\mu(B_{2\rho})^{1/n}=2^{-n}K_0^{-1}\rho(\mu)_{B_{2\rho}}^{1/n},
\]
where $\sigma_n=|B_1|$. On the other hand, we have 
\[
\rho^2(\mu)_{B_{\rho}}^{1/n}=\sigma_n^{-1/n}\rho\mu(B_{\rho})^{1/n}\leq \sigma_n^{-1/n}\rho\mu(B_{2\rho})^{1/n}=2\rho^2(\mu)_{B_{2\rho}}^{1/n}.
\]
Now, by taking $K=\max\{2^nK_0K_1, 2K_1\}$, we infer from the last two estimates and \eqref{condition K} that
\[
K^{-1}\rho(\mu)_{B_{2\rho}}^{1/n}|y|\leq s \leq K\rho^2(\mu)_{B_{2\rho}}^{1/n}.
\]
Hence,  \eqref{slant-cyl} holds.

\smallskip
Now, choose $\hat \epsilon_2=\epsilon_1(n, \nu, K_0, K)\in (0,1)$ and $\hat \beta_2=\beta_1(n, \nu, K_0, K)$, where $\epsilon_1$ and $\beta_1$ are the numbers defined in Lemma \ref{slant-cylin-lemma}. Then, applying Lemma \ref{slant-cylin-lemma} to $u$ in $V_{\rho}(Y)$ and using the fact that $V_{\rho}(Y)\subset U$, we obtain
\[
\sup_{B_{\rho/2}(y)}u(\cdot, s)\leq \hat \beta_2\sup_{\partial_p V_{\rho}(Y)}u^+\leq \hat\beta_2 \sup_{U}u^+.
\]
Since $Y=(y, s)$ is arbitrary in $B_{\rho/2}\times \big(\rho^2(\mu)_{B_\rho}^{1/n}, K_1\rho^2(\mu)_{B_\rho}^{1/n}\big)$, the assertion of the lemma follows. The lemma is proved.
\end{proof}
%=====

We prepare the following result related to our weighted cylinders, which will also be used in the proof of Lemma \ref{alternative-lemma}.
\begin{lemma}\label{inclusion-lemma}
Let $r>0$ and $Y=(y, s) \in \mathbb{R}^{n+1}$. Define the set
\[
 Q=B_{r}(y)\times\big(s- \frac{7}{2}r^2 (\mu)_{B_{2r}(y)}^{1/n},s-3 r^2 (\mu)_{B_{2r}(y)}^{1/n}\big).
\]
For each $\kappa_0>0$, there exist constants $\eta_0=\eta_0(n, K_0, \kappa_0)>0$ and $\tilde{r} >(1+ \eta_0)r$ such that the cylinder
\[ C'=B_{\tilde r} (y) \times \big(s- \frac{7}{2}r^2 (\mu)_{B_{2r}(y)}^{1/n},  
s- \big(3 - \frac{\kappa_0}{4}\big)r^2(\mu)_{B_{2r}(y)}^{1/n}\big)
\] 
 satisfies
\begin{equation}  \label{cyl-inclu-re-1}
Q \subset C', \quad \mu(C'\setminus Q) \leq \kappa_0 \mu(Q),%\quad \text{and}\quad \tilde r\geq (1+\eta_0)r
\end{equation}
and
\begin{equation}\label{dist cylinder c-1}
\textup{dist}\Big(Q, \partial C' \setminus \big(B_{\tilde r}(y)\times 
\big\{ s-\frac{7}{2}r^2(\mu)_{B_{2r}(y)}^{1/n}\big\}\big)\Big)  \geq  \min\Big\{\eta_0r, \frac{\kappa_0}{4}r^2(\mu)_{B_{2r}(y)}^{1/n}\Big\}.
\end{equation}
\end{lemma}

\begin{proof} 
Without loss of generality, we assume that $Y=(y, s) =(0,0)$. Then,  
\[ C'=B_{\tilde r}\times \big(-\frac{7}{2}r^2(\mu)_{B_{2r}}^{1/n},  \big(-3+\frac{\kappa_0}{4}\big)r^2(\mu)_{B_{2r}}^{1/n}\big).
\] 
Observe that if $\tilde r>r$ then $Q \subset C'$.  For a given $\kappa_0>0$, we look for $\tilde r>r$ so that
\begin{equation}\label{cylinder c-1-1}
\mu(C')=(1+\kappa_0)\mu(Q).
\end{equation}
Note that by \eqref{cylinder measure formula}, we have
\begin{align*}
\mu(Q) &= \frac{1}{4}\sigma_n^{-1/n} r \mu(B_r) \mu(B_{2r})^{1/n}\qquad \text{and}\\
\mu(C') &= \frac{1}{2}\sigma_n^{-1/n} \big(\frac{1}{2}+\frac{\kappa_0}{4}\big) r \mu(B_{\tilde{r}}) \mu(B_{2r})^{1/n},
\end{align*}
where $\sigma_n=|B_1|$. Then
\[
\frac{\mu(Q)}{\mu(C')} = \frac{2\mu(B_r)}{(2+\kappa_0)\mu(B_{\tilde{r}})},
\]
which implies that \eqref{cylinder c-1-1} is equivalent to
\begin{equation} \label{ball-10-052024}
\frac{\mu(B_r)}{\mu(B_{\tilde{r}})} =\frac{2+\kappa_0 } {2+ 2\kappa_0} <1.
\end{equation}
Let us denote $g(\theta) = \mu(B_r)/\mu(B_\theta)$ for $\theta\in [r,\infty)$. As $g$ is continuous, monotonically decreasing, 
\[
g(r) =1, \quad \text{and} \quad \lim_{\theta\rightarrow \infty} g (\theta) =0,
\]
we obtain the existence of $\tilde{r} >r$ such that $g(\tilde{r}) = \frac{2+\kappa_0 } {2+ 2\kappa_0}<1$, and then \eqref{ball-10-052024} holds.

\smallskip
Next, we prove that there is $\eta_0 = \eta_0(n, K_0, \kappa_0)>0$ so that $\tilde{r} \geq (1+\eta_0) r$. By \eqref{ball-10-052024}, it follows that
\[
\frac{\mu(B_{\tilde r}\setminus B_r)}{\mu(B_{\tilde r})}=1- \frac{\mu(B_r)}{\mu(B_{\tilde{r}})}  = \frac{\kappa_0}{2+2\kappa_0}.
\]
On the other hand, noting that $[\mu]_{A_{n+1}}=[\omega]_{A_{1+\frac{1}{n}}}^n\leq K_0^n$ and applying Lemma \ref{reverse property} to $\mu$ with $S=B_{\tilde r}\setminus B_r$ and $B=B_{\tilde r}$, we have
\[
\frac{\mu(B_{\tilde r}\setminus B_r)}{\mu(B_{\tilde r})}\leq N\Big(\frac{{\tilde r}^n-r^n}{{\tilde r}^n}\Big)^{\lambda_0} = N\Big[1-\Big(\frac{r}{\tilde r}\Big)^n\Big]^{\lambda_0},
\]
where $N=N(n, K_0) > 0$ and $\lambda_0=\lambda_0(n, K_0)>0$. Therefore, combining the previous two estimates, we have
\[
\frac{\kappa_0}{2+2\kappa_0}\leq N\Big[1-\Big(\frac{r}{\tilde r}\Big)^n\Big]^{\lambda_0},
\]
which implies
\[
\tilde{r} \geq \Big[1-\Big(\frac{\kappa_0}{N(2+2\kappa_0)} \Big)^{1/\lambda_0} \Big]^{-1/n}r>r.
\]
Hence, 
\begin{equation}\label{eta_0-1}
\tilde r \geq (1+\eta_0)r, \quad \text{for} \quad \eta_0 = \Big[1-\Big(\frac{\kappa_0}{N(2+2\kappa_0)} \Big)^{1/\lambda_0} \Big]^{-1/n} -1 >0.
\end{equation}
From \eqref{cylinder c-1-1} and the previous estimate, we obtain \eqref{cyl-inclu-re-1}. 
Moreover, the formula \eqref{dist cylinder c-1} immediately follows from the explicit formula of $C'$, \eqref{cyl-inclu-re-1}, and \eqref{eta_0-1}. The lemma is proved.
\end{proof}
The following lemma is derived from the first and second growth lemmas (Proposition ~\ref{growth-1} and  \ref{second-growth-lemma}), and the covering lemma (Lemma \ref{covering-1}).  The proof of the lemma is the most technical one in the paper.
\begin{lemma} \label{alternative-lemma} For $\nu \in (0,1)$,  $\delta \in (0,1)$, and $K_0\in [1,\infty)$, there exist $\bar \epsilon_3 =\bar \epsilon_3(n,\nu, K_0, \delta) \in (0,1)$, $\beta_0= \beta_0(n, \nu, K_0,\delta) \in (0,1)$, and $\kappa = \kappa(n, \nu, K_0)>0$ such that the following assertion holds. Assume that \eqref{elliptic} holds in $C_{2r,\mu}(Y)$ with some $r>0$ and $Y=(y,s) \in \mathbb{R}^{n+1}$,  and assume also that $\omega\in A_{1+\frac{1}{n}}$ satisfies
\begin{equation} \label{omega-cond-3-growth}
 [\omega] _{A_{1+\frac{1}{n}}}\leq K_0 \quad\text{ and} \quad [[\omega]]_{\textup{BMO}(B_{2r}(y), \omega)}\leq \bar \epsilon_3.
\end{equation} 
Then, for every $u \in C^{2,1}(C_{2r,\mu}(Y)) \cap C(\overline{{C}_{2r,\mu}(Y)})$ satisfying 
\begin{equation}\label{density condition}
\mathcal{L} u \leq 0 \quad \text{in} \quad C_{2r,\mu}(Y), \quad \text{and} \quad \mu\big(\{u >0 \} \cap Q \big) \leq \delta \mu(Q),
\end{equation}
we have 
\begin{equation} \label{alternative-conclusion}
\begin{aligned}\left\{
\begin{split}
& \textup{either (i)} \quad \sup_{C_{r,\mu}(Y)} u^+ \leq \beta_0 \sup_{C_{2r, \mu}(Y)} u^+,  \\
& \textup{or \ \quad  (ii)}  \quad \mu\big(  \{u\leq \beta_0M\}\cap Q \big) \geq (1+ \kappa) \mu(\Gamma),
\end{split} \right.
\end{aligned}
\end{equation}
where 
\begin{align*}
& Q=B_{r}(y)\times\big(s- \frac{7}{2}r^2 (\mu)_{B_{2r}(y)}^{1/n},s-3 r^2 (\mu)_{B_{2r}(y)}^{1/n}\big), \\
& M = \displaystyle{\sup_{C_{2r, \mu}(Y)} u^+}, \quad \text{and} \quad \Gamma = \{u \leq 0 \} \cap Q.
\end{align*}
\end{lemma}
%\smallskip
%\noindent
\begin{proof}  We follow the technique introduced by N. V. Krylov and M. V. Safonov, see \cites{KrySa-2, FeSa}. 
By taking a linear translation, we may assume that $Y=(y,s)=(0,0)$. 
In this setting, we have
\begin{align*}
& C_{2r, \mu}(Y) = C_{2r, \mu} =B_{2r} \times(- 4r^2 (\mu)_{B_{2r}}^{1/n}, 0), \\
& Q=B_{r} \times\big(- \frac{7}{2}r^2 (\mu)_{B_{2r}}^{1/n}, -3 r^2 (\mu)_{B_{2r}}^{1/n}\big).
\end{align*}
Observe that the assertion of the proposition is trivial when $M=0$. From now on, we assume $M>0$. By \eqref{density condition}, we also have
\begin{equation}\label{Gamma-measure}
\mu(\Gamma)= \mu(Q)-\mu\big(\{u >0 \} \cap Q \big)\geq (1-\delta)\mu(Q).
\end{equation}
Now, for setting the proof, let $N=N(n,\nu,K_0)>0$ be the number defined in \eqref{growth-1-est} in Proposition~\ref{growth-1}. 
We choose $\epsilon_0=\epsilon_0(n,\nu, K_0)$ and $\delta_0=\delta_0(n,\nu, K_0)\in(0,1)$ to be sufficiently small such that 
\begin{equation} \label{epsilon-06-16}
N\Big(\delta_0^{\frac{1}{n+1}} +\epsilon_0\Big)\leq \frac{1}{2}.
\end{equation}
Note that from the definition of $\Gamma$, it is obvious that $\Gamma$ is bounded. Then, in order to apply the  covering lemma (Lemma \ref{covering-1}),  let  $\A$ be the family of open cylinders defined in \eqref{cylinders set} corresponding to the set $\Gamma$ and the number $\delta_0$. Also, let  $E$
be the set defined in \eqref{set E}. From the construction, we have
\begin{equation} \label{growth-3-E-construc}
\mu(C\cap \Gamma) \geq (1-\delta_0) \mu(C), \quad \forall \ C \in \mathcal{A}, \quad \text{and} \quad E = \bigcup_{C\in \mathcal{A}} C.
\end{equation}
From this and the definition of $\Gamma$, we infer that 
\begin{equation} \label{3rd-growth-A-class}
 \mu\big( C \cap \big\{u>0 \big\} \big) \leq \mu(C\setminus \Gamma) \leq \delta_0 \mu(C), \quad \forall \ C \in \mathcal{A}.
\end{equation}
Also, observe that as $\delta_0$ is sufficiently small and $\Gamma \subset Q$, it follows from the definitions of $\mathcal{A}$ and $Q$ that
\begin{equation} \label{C-inclusion-A-10-02}
C \subset C_{2r, \mu}, \quad \forall \ C \in \mathcal{A}.
\end{equation}
In addition, it follows from Lemma \ref{covering-1} that
\begin{equation}\label{measure est 1}
\mu(E)\geq q_0\mu(\Gamma),\quad \text{where}\quad q_0=1+3^{-(n+1)^2-1}K_0^{-n-1}\delta_0>1.
\end{equation}
Here $q_0 = q_0(n,\nu, K_0)$ because $\delta_0$ depends on $n$,\ $\nu$, and $K_0$.  We then set 
\begin{equation} \label{kappa-choice-3}
\kappa=\kappa(n,\nu,K_0)=3^{-(n+1)^2-2}K_0^{-n-1}\delta_0\in (0,1) \quad \text{so that} \quad  q_0=1+3\kappa. 
\end{equation}
Then, take
\begin{equation}\label{K_1}
K_1=K_1(n,\nu, K_0)=5+2\kappa^{-1}>1.
\end{equation}
With this $K_1$, let $\hat E$ be defined as in \eqref{set hat E}. Then, by \eqref{hat-E-E}, \eqref{measure est 1}, and the definition of $\kappa$ in \eqref{kappa-choice-3}, it follows that
\begin{equation}\label{measure hat E}
\mu(\hat E)\geq q_1\mu(E)\geq q_0q_1\mu(\Gamma)=(1+2\kappa)\mu(\Gamma),
\end{equation}
because
\[
q_1=\frac{K_1-1}{K_1+1}=\frac{1+2\kappa}{1+3\kappa}.
\]

\smallskip
Next, we introduce the constants $\bar{\epsilon}_3$ and $\beta_0$ appearing in the statement of the proposition. Let
\begin{equation}\label{choice-epsilon-3}
\bar \epsilon_3 =\bar{\epsilon}_3(n, \nu, K_0, \delta) = \min\big \{\epsilon_0, \hat\epsilon_2,  \epsilon_2\big\},
\end{equation}
where $\epsilon_0 =\epsilon_0(n, \nu, K_0)$ is defined in \eqref{epsilon-06-16}, $\hat\epsilon_2=\hat\epsilon_2(n, \nu, K_0, K_1)$ is the constant defined in Lemma \ref{up-lemma-5}, and $\epsilon_2 = \epsilon_2(n, \nu, K_0, 4\eta_1^{-1}, 1)$ is the constant defined in Proposition \ref{second-growth-lemma}. Here, note that $K_1=K_1(n,\nu, K_0)$ is defined in \eqref{K_1} and $\eta_1 = \eta_1(n, \nu, K_0, \delta)$ is the number defined in \eqref{eta-1} below. We also set
\begin{equation} \label{choice-beta-epsilon-2}
\beta_0=\beta_0(n,\nu, K_0, \delta)=\max\Big\{\frac{1+\hat \beta_2}{2}, \frac{1+\beta_2}{2}\Big\}\in (0,1),
\end{equation}
where $\hat \beta_2=\hat\beta_2(n,\nu, K_0, K_1)$ is the constant defined in Lemma \ref{up-lemma-5}, and $\beta_2=\beta_2(n,\nu, K_0, 4\eta_1^{-1}, 1)\in (0,1)$ is defined in Proposition \ref{second-growth-lemma}.%, in which $\eta_1(n, \nu, K_0, \delta)$ is the number defined in \eqref{eta-1} below.

\smallskip
We now prove \eqref{alternative-conclusion} with the choices of $\kappa,  \bar \epsilon_3, \beta_0$ defined in \eqref{kappa-choice-3}, \eqref{choice-epsilon-3}, and \eqref{choice-beta-epsilon-2}. We first note that by \eqref{3rd-growth-A-class}, \eqref{C-inclusion-A-10-02}, and the choice of $\bar \epsilon_3$, it follows from Proposition \ref{growth-1} that
\begin{equation}  \label{C-A-prop-growth}
\sup_{C_{\theta/2, \mu}(X)} u^+ \leq \frac{1}{2} \sup_{C_{\theta, \mu}(X)} u \leq \frac{M}{2}, \quad \forall \ C_{\theta, \mu}(X) \in \mathcal{A}.
\end{equation}
Then, for such $C=C_{\theta, \mu}(X)$ in \eqref{C-A-prop-growth}, if $\hat{C} \subset U\subset C_{2r, \mu}$,  it follows from \eqref{omega-cond-3-growth},  \eqref{choice-epsilon-3}, \eqref{C-A-prop-growth}, and Lemma \ref{up-lemma-5} that
\[
\sup_{\hat{C}} \big( u -\frac{M}{2}\big)  \leq \hat{\beta}_2 \sup_{C_{2r,\mu}}(u-\frac{M}{2})\leq \frac{\hat \beta_2}{2}M,
\]
where $\hat{C}$ is the corresponding cylinder of $C$ defined in \eqref{hat cylinder} and $U= U_{\theta, \mu}(X)$ is defined in \eqref{cylinder U}. From this, \eqref{choice-beta-epsilon-2},  we obtain the following important estimate on the growth of the solution $u$
\begin{equation} \label{hat-C-growth}
\sup_{\hat{C}} u \leq \beta_0 M, \quad \text{if} \quad \ C \in \mathcal{A}:    \ \hat{C} \subset U\subset C_{2r, \mu}.
\end{equation}

\smallskip
Next, for the given $Q$ in the statement of the lemma, let 
\begin{equation} \label{kappa-eta-con-def-0905}
\kappa_0 =\kappa(1-\delta) \in (0,1) \quad \text{and} \quad \eta_0=\eta_0(n, \nu, K_0, \delta)>0,
\end{equation}
where $\kappa=\kappa(n,\nu, K_0)$ is defined in \eqref{kappa-choice-3}, and $\eta_0(n, \nu, K_0, \delta)$ is defined in \eqref{eta_0-1}. By Lemma \ref{inclusion-lemma}, we see that the cylinder 
\begin{equation} \label{C-prime-def-0905}
C'=B_{\tilde r}\times \big( -\frac{7}{2}r^2(\mu)_{B_{2r}}^{1/n},\big(-3+\frac{\kappa_0}{4}\big)r^2(\mu)_{B_{2r}}^{1/n}\big)
\end{equation}
satisfies
\begin{equation}\label{cylinder c-1}
Q \subset C', \quad \mu(C')\leq \big(1+\kappa_0\big)\mu(Q), \quad \tilde r \geq (1+\eta_0) r,
\end{equation}
and
\[
\textup{dist}\Big(Q, \partial C' \setminus \big(B_{\tilde r}\times \big\{ -\frac{7}{2} r^2(\mu)_{B_{2r}}^{1/n}\big\}\big)\Big)  \geq  \min\Big\{\eta_0r, \frac{\kappa_0}{4}r^2(\mu)_{B_{2r}}^{1/n}\Big\}.
\]
Note that we can assume without loss of generality that $C' \subset C_{2r, \mu}$ because we can choose $\delta_0$ small enough so that $\kappa$ defined in \eqref{kappa-choice-3} is sufficiently small, and this makes $\kappa_0$ sufficiently small.

% as otherwise we simply take $C'=C_{2r, \mu}$. 
\smallskip
The proof is then split into the following two cases based on the relation between $\hat{E}$ and $C'$.

\smallskip
\noindent  
\textbf{Case 1: $\mu(\hat E\setminus C')>0$.}  Because of this, and by the definition of $\hat E$,  we can find a cylinder $C= C_{r_0,\mu}(Y_0) \in \mathcal{A}$ for some $r_0>0$ and $Y_0 =(y_0, s_0) \in \mathbb{R}^{n+1}$ such that 
\begin{equation} \label{hat-C-out-C-1}
\hat C \setminus C'\neq \emptyset, 
\end{equation}
where
\[
\hat C = \hat C_{r_0,\mu}(Y_0)=B_{r_0}(y_0)\times \big(s_0+r_0^2(\mu)^{1/n}_{B_{r_0}(y_0)}, s_0+K_1r_0^2(\mu)^{1/n}_{B_{r_0}(y_0)}\big)
\]
is defined in \eqref{hat cylinder}. Based on \eqref{hat-C-out-C-1} and the definition of $C'$ in \eqref{C-prime-def-0905}, we infer that the cylinder $C$ is not too small. Specifically, we claim that
\begin{equation} \label{eta-1}
 r_0 \geq \eta_1 r \quad \text{for} \quad \eta_1= \min\Big\{\frac{\eta_0}{2},  \frac{\kappa(1-\delta)}{8(1+K_1)},  \frac{1}{2}\Big\}.
\end{equation}
Here, it is important to recall that $\kappa$ is defined in \eqref{kappa-choice-3}, $K_1$ is defined in \eqref{K_1}, and $\eta_0$ is defined in \eqref{kappa-eta-con-def-0905}. As such, $\eta_1 = \eta_1(n, \nu, K_0, \delta) >0$.

\smallskip
To verify \eqref{eta-1}, we only need to consider the case $r_0 \leq r/2$, as otherwise \eqref{eta-1} is trivial. We recall that $\Gamma\subset Q$, and by \eqref{growth-3-E-construc}, we also have $\mu(C\cap\Gamma)>0$. This implies that $B_{r_0}(y_0) \cap B_r \not= \emptyset$ and
\[
 \big(s_0- r_0^2(\mu)^{1/n}_{B_{r_0}(y_0)}, s_0\big) \cap  \big(- \frac{7}{2}r^2 (\mu)_{B_{2r}}^{1/n}, -3 r^2 (\mu)_{B_{2r}}^{1/n}\big) \not= \emptyset.
\]
From this, and since $r_0\leq r/2$, we have 
\begin{equation}\label{intersect}
B_{r_0}(y_0)\subset B_{2r} \quad \text{and} \quad -\frac{7}{2}r^2(\mu)^{1/n}_{B_{2r}}<s_0<-3r^2(\mu)^{1/n}_{B_{2r}}+r_0^2(\mu)_{B_{r_0}(y_0)}^{1/n}.
\end{equation}
Now, from \eqref{hat-C-out-C-1}, we note that either $B_{r_0}(y_0) \setminus B_{\tilde{r}} \not=\emptyset$ or the top of $\hat{C}$ must be above the top of $C'$. Because of this,  and by \eqref{intersect}, and the definition of $C'$ in \eqref{C-prime-def-0905}, we infer that
\begin{equation}\label{radius-2}
\text{either} \quad 2r_0\geq \eta_0r \quad \text{or} \quad (1+K_1)r_0^2(\mu)^{1/n}_{B_{r_0}(y_0)}\geq \frac{\kappa_0}{4}r^2
(\mu)_{B_{2r}}^{1/n}.
\end{equation}
Note that if the former in \eqref{radius-2} holds, then \eqref{eta-1} follows.  On the other hand, if the latter condition in \eqref{radius-2} holds, then by combining this with the fact that $\mu(B_{r_0}(y_0)) \leq \mu(B_{2r})$ due to \eqref{intersect}, and the fact that $\kappa_0=\kappa(1-\delta)$, we obtain
\[
r_0\geq \frac{\kappa(1-\delta)}{8(1+K_1)}r,
\]
which also implies  \eqref{eta-1}.  Hence, in both cases, we both have \eqref{eta-1}.
%-====

\smallskip
On the other hand, from \eqref{C-inclusion-A-10-02}, the definition of $\mathcal{A}$, we can translate $C_{r_0,\mu}(Y_0)$ in time direction if needed and assume that
\begin{equation} \label{06-17-1}
 s_0 \leq -3r^2(\mu)_{B_{2r}}^{1/n}.
\end{equation}
%%%%%%%%%%%%%%%%%%%%%%%%%%%%%%%%%%%%%%%
%%%%%%%%%%%%%%%%%%%%%%%%%%%%%%
Also, it follows from \eqref{C-A-prop-growth} that
\begin{equation}  \label{case-1 u}
\sup_{B_{r_0/2}(y_0)}u(\cdot, s_0)\leq\sup_{C_{r_0/2,\mu}(Y_0)}u\leq \frac{1}{2}\sup_{C_{r_0,\mu}( Y_0)}u\leq \frac{1}{2}M.
\end{equation}
Then, from the choice of $\bar \epsilon_3$ in \eqref{choice-epsilon-3}, along with \eqref{06-17-1} and \eqref{case-1 u}, we apply Proposition \ref{second-growth-lemma} to $u-\frac{1}{2}M$ in $C_{2r, \mu}$ with $h=1$, $\tau=s_0$, and $\rho=\frac{1}{2}\eta_1r\leq \frac{1}{4}r$. Noting that  condition \eqref{tau-condition} holds for $\tau$ defined in \eqref{06-17-1}, we obtain
\[
\sup_{C_{r,\mu}}(u-\frac{1}{2}M)\leq \beta_2 \sup_{C_{2r,\mu}}(u-\frac{1}{2}M)\leq \frac{\beta_2}{2}M,
\]
where $\beta_2=\beta_2(n,\nu, K_0, 4\eta_1^{-1}, 1)\in (0,1)$ is the number defined in Proposition \ref{second-growth-lemma}.  It then follows from this and the choice of $\beta_0$ in \eqref{choice-beta-epsilon-2} that
\[
\sup_{C_{r,\mu}}u \leq \beta_0 M.
\]
Thus, we obtain (i)  in \eqref{alternative-conclusion}. 
\\ \ \\
\noindent
\textbf{Case 2: $\mu(\hat E\setminus C')=0$.} This implies $\hat E\subset C' \subset C_{2r, \mu}$. Because of this, and due to the proof in {\bf Case 1}, we can assume without loss of generality that every $C \in \mathcal{A}$ has a radius smaller than $r/4$. Hence, for any arbitrary cylinder $C_{\rho,\mu}(X_0)\in\A$ with $X_0=(x_0,t_0)\in \R^{n+1}$ and $\rho>0$, we have $\rho \in (0, r/4)$ and
\[ 
\hat C_{\rho,\mu}(X_0)  \subset U \subset C_{2r,\mu},
\]
where
\[
\hat C_{\rho,\mu}(X_0)=B_{\rho}(x_0)\times \big(t_0+{\rho}^2(\mu)_{B_{\rho}(x_0)}^{1/n}, t_0+K_1\rho^2(\mu)_{B_{\rho}(x_0)}^{1/n}\big),
\] 
and $U = U_{\rho, \mu}(X_0)$ is the cylinder corresponding to $\hat C_{\rho,\mu}(X_0)$ that we defined in \eqref{cylinder U}.
Then, by \eqref{hat-C-growth}, we have 
\[ 
\sup_{\hat C_{\rho,\mu}(X_0)} u^+ \leq \beta_0 M.
\]
Therefore,
\[
\sup_{\hat{E}} u^+ \leq \beta_0 M.
\]
Let us denote $\hat\Gamma =\hat E\cap Q$. By \eqref{Gamma-measure}, \eqref{measure hat E}, \eqref{kappa-eta-con-def-0905}, and \eqref{cylinder c-1}, we have
\begin{equation*}\label{case-2 Gamma}
\begin{aligned}
\mu(\hat\Gamma) &=\mu(\hat E\cap Q) \geq \mu(\hat E)-\mu(C'\setminus Q)\\
&\geq (1+2\kappa)\mu(\Gamma)-\kappa(1-\delta)\mu(Q)\\
&\geq  (1+2\kappa)\mu(\Gamma)-\kappa\mu(\Gamma)\\
&=(1+\kappa)\mu(\Gamma).
\end{aligned}
\end{equation*}
%Recall that $\Gamma=\{u\leq 0\}\cap Q$, 
It then follows from the last two estimates that
\[ 
\mu\big(\{u\leq \beta_0M\}\cap Q\big)\geq \mu(\hat\Gamma)\geq (1+\kappa)\mu(\Gamma).
\] 
From this, we obtain (ii) of \eqref{alternative-conclusion}. 
The proof is completed.
\end{proof}
%====
The following result, which is known as the third growth lemma, is the main result of the subsection.
\begin{proposition}[Third Growth Lemma] \label{third-growth} For $\nu \in (0,1)$, $\delta \in (0, 1)$, and $K_0\in [1,\infty)$, there exist $\epsilon_3=\epsilon_3(n,\nu, K_0, \delta) \in (0,1)$ and $\beta_3= \beta_3(n, \nu, K_0, \delta) \in (0,1)$ such that the following assertion holds. Assume that \eqref{elliptic} holds in $C_{2r,\mu}(Y)$ with some $r>0$ and $Y=(y,s) \in \mathbb{R}^{n+1}$,  and assume also that $\omega\in A_{1+\frac{1}{n}}$ satisfies
\begin{equation} \label{omega-cond-3-growth-06-16}
 [\omega] _{A_{1+\frac{1}{n}}}\leq K_0 \quad\text{ and} \quad [[\omega]]_{\textup{BMO}(B_{2r}(y), \omega)}\leq \epsilon_3.
\end{equation} 
Then, for every $u \in C^{2,1}(C_{2r,\mu}(Y)) \cap C(\overline{{C}_{2r,\mu}(Y)})$ satisfying 
\begin{equation}\label{density condition - grow3}
\mathcal{L} u \leq 0 \quad \text{in} \quad C_{2r,\mu}(Y), \quad \text{and} \quad \mu\big(\{u >0 \} \cap Q \big) \leq \delta \mu(Q),
\end{equation}
where 
\[ 
 Q=B_{r}(y)\times\big(s- \frac{7}{2}r^2 (\mu)_{B_{2r}(y)}^{1/n},s-3 r^2 (\mu)_{B_{2r}(y)}^{1/n}\big),
\]
we have
\begin{equation} \label{growth-3-conclusion}
\sup_{C_{r, \mu}(Y)} u^+ \leq \beta_3 \sup_{C_{2r, \mu}(Y)} u^+.
\end{equation}
\end{proposition}
\begin{proof}  Let $\epsilon_3 =\bar{\epsilon}_3$, where $\bar{\epsilon}_3$ is the number defined in Lemma \ref{alternative-lemma} and we prove the proposition with this choice of $\epsilon_3$. We apply Lemma \ref{alternative-lemma} iteratively to obtain \eqref{growth-3-conclusion}. By Lemma \ref{alternative-lemma}, there exist $\beta_0=\beta_0(n,\nu, K_0, \delta)\in (0,1)$ and $\kappa=\kappa(n,\nu, K_0)>0$ such that
\begin{equation}
\begin{aligned}\label{iteration-u-conclude}
\left\{
\begin{split}
& \textup{either (i)} \quad  u \leq \beta_0 M \quad \text{in}\quad C_{r,\mu}(Y), \\
& \textup{or  \ \quad (ii)}  \quad \mu\big(  \Gamma_1 \big) \geq (1+ \kappa) \mu(\Gamma),
\end{split} \right.
\end{aligned} 
\end{equation}
where
\[
M = \sup_{C_{2r, \mu}(Y)} u^+, \quad \Gamma_1 = \{u\leq \beta_0M\}\cap Q, \quad \text{and} \quad \Gamma =\{ u \leq 0 \} \cap Q.
\]
If (i) in \eqref{iteration-u-conclude} holds true, then we take $\beta_3 =\beta_0$ and obtain \eqref{growth-3-conclusion}, and the proof of the proposition is completed. Otherwise, when (ii) in \eqref{iteration-u-conclude} holds true, we define 
\[  
u_1=u-\beta_0M\quad \text{in}\quad \overline{C_{2r,\mu}(Y)}.
\]
Because of \eqref{omega-cond-3-growth-06-16} and \eqref{density condition - grow3}, we see that $u_1$ satisfies the conditions in Lemma \ref{alternative-lemma}. Then, by applying Lemma \ref{alternative-lemma} for $u_1$, we infer that
\begin{equation} \label{step-2-iteration}
\left\{
\begin{split}
& \text{either (i-1)} \quad u_1\leq \beta_0M_1 \quad \text{in}\quad C_{r,\mu}(Y), \\
& \text{or \ \quad (ii-1)} \quad \mu\big(\Gamma_2 \big) \geq (1+\kappa)\mu(\Gamma_1),
\end{split} \right.
\end{equation}
where
\[
\Gamma_2 = \{u_1\leq \beta_0M_1\} \cap Q \quad \text{and} \quad M_1 = \sup_{C_{2r,\mu}(Y)} u_1 = (1-\beta_0)M.
\]
Again, if (i-1) in \eqref{step-2-iteration} holds, 
we obtain 
\[
u  \leq \big[1- (1-\beta_0)^2\big]M \quad \text{in} \quad C_{r,\mu}(Y).
\]
Then \eqref{growth-3-conclusion} follows, and the process is completed. Otherwise, when (ii-1)  in \eqref{step-2-iteration} holds. As (ii) in  \eqref{iteration-u-conclude} also holds, we get
\[
\mu(\Gamma_2) \geq (1+\kappa)^2 \mu(\Gamma).
\]
Then, we set
\[ u_2 = u_1 -\beta_0 M_1 = u - \big[1- (1-\beta_0)^2\big]M,
\]
and
\[  
\Gamma_3 =  \{u_2\leq \beta_0M_2\} \cap Q, \quad \text{and} \quad M_2= \sup_{C_{2r,\mu}(Y)} u_2.
\]
Assume that we can keep doing this up to some step $m \in \mathbb{N}$. Then, we get the sequences $\{u_k\}_{k=1}^m$ and $\{\Gamma_k\}_{k=1}^{m+1}$, defined as
\begin{equation}\label{induction-1}
u_k = u - \big[1- (1-\beta_0)^k \big] M \quad \text{and} \quad \Gamma_{k+1} = \{u_k \leq \beta_0M_k\} \cap Q
\end{equation}
with  $M_k =\displaystyle{\sup_{C_{2r, \mu}(Y)}} u_k$ for $k=1,2,\ldots, m$. Moreover, we also have a series of alternatives 
\begin{equation} \label{iteration-k-step}
\left\{
\begin{split}
& \text{either (i-k)} \quad u_k \leq \beta_0 M_k \quad \text{in} \quad C_{r, \mu}(Y), \\
& \text{or \ \quad (ii-k)} \quad \mu(\Gamma_{k+1}) \geq (1+\kappa)^{k+1} \mu(\Gamma), 
\end{split} \right.
\end{equation}
in which neither $\textup{(i-k)}$ hold but all $\textup{(ii-k)}$ hold for $k = 1, 2,\ldots, m-1$. We note $\Gamma_{k} \subset Q$ holds for all $k$. Moreover, from  (ii-k) in \eqref{iteration-k-step}, it follows that
\[
k +1 \leq \log_{1+\kappa} \big[\mu(Q)/ \mu(\Gamma) \big] \leq -\log_{1+\kappa}(1-\delta),
\]
where we used \eqref{Gamma-measure} in the last step. Therefore, the process must stop at some $k=m \leq k_0$, where $k_0$ is the largest natural number satisfying
\[
k_0 \leq-\log_{1+\kappa}(1-\delta).
\]
This means that (i-k) in \eqref{iteration-k-step} holds when $k=m$. We note that from the definition of $u_m$ in \eqref{induction-1}, we infer that
\begin{align*}
u \leq [1-(1-\beta_0)^{m+1}] M\leq [1-(1-\beta_0)^{k_0+1}] M\quad \text{in} \quad C_{r, \mu}.
\end{align*}
This implies \eqref{growth-3-conclusion} with 
\[
\beta_3(n,\nu, K_0, \delta) = 1-(1-\beta_0)^{k_0+1} \in (0,1),
\]
and the proof of the proposition is completed.
\end{proof}
%=====
To conclude the subsection, we state and prove the following result, which is a simple corollary of Proposition \ref{third-growth}.
\begin{corollary}\label{Growth corollary} Assume as Proposition \textup{\ref{third-growth}}. Then, for all $u \in C^{2,1}(C_{2r,\mu}(Y)) \cap C(\overline{{C}_{2r,\mu}(Y)})$ satisfying $u\geq 0$,\ $\mathcal{L} u \geq 0$ on $C_{2r,\mu}(Y)$, and
\begin{equation} \label{corol-0907-cond}
\mu\big(\{u \geq1 \} \cap Q \big) \geq (1-\delta) \mu(Q)
\end{equation}
with $ 
 Q=B_{r}(y)\times\big(s- \frac{7}{2}r^2 (\mu)_{B_{2r}(y)}^{1/n},s-3 r^2 (\mu)_{B_{2r}(y)}^{1/n}\big)$, it holds that
\[
\inf_{C_{r, \mu}(Y)}u\geq 1-\beta_3>0.
\]
\end{corollary}
%====
\begin{proof}
Set $v=1-u$, then $v$ satisfies all the conditions in Proposition \ref{third-growth}. Therefore,
\[
\sup_{C_{r, \mu}(Y)} v \leq \beta_3 \sup_{C_{2r, \mu}(Y)} v.
\]
Hence,
\[
\inf_{C_{r, \mu}(Y)}u\geq 1-\beta_3+\beta_3\inf_{C_{2r, \mu}(Y)} u\geq 1-\beta_3>0.
\]
\end{proof}
\section{Proofs of the main results} \label{proof-section}
\subsection{Proof of Theorem~\protect{\ref{Harnack inequality}}} It is N. V. Krylov and M. V. Safonov \cite{KrySa-2} who first proved  
the Harnack inequality for the solutions to second order linear equations in non-divergence form with uniform elliptic and bounded coefficients,  using some ideas originated from E. M. Landis \cite{E.M.Landis}. We follow the arguments presented in \cite{FeSa}, which rely on a distance function, Corollary \ref{Growth corollary}, and Corollary \ref{prop-up-lemma}.

\begin{proof} For settings, we let
\begin{equation} \label{gamma-def-0908}
\gamma = \gamma(n, \nu, K_0, \frac{1}{16})>0
\end{equation}
be the constant defined in Corollary~\ref{prop-up-lemma}. Also, let  $\delta_0=\delta_0(n,\nu,K_0) \in (0,1)$ and $\epsilon_0=\epsilon_0(n,\nu,K_0)\in (0,1)$ be sufficiently small  so that  
 \begin{equation}\label{epsilon using 1}
 NK_0\delta_0^{\frac{1}{n+1}} \le  2^{-\gamma-2}
 \quad \text{and}\quad 
 NK_0\epsilon_0 \le 2^{-\gamma-2},
 \end{equation}
where $N=N(n,\nu)>0$ is the constant defined in the claim \eqref{claim-1} in the proof of Proposition~\ref{growth-1}. In addition, let 
\begin{equation}\label{def of delta-1}
\delta_1(n, \nu, K_0) = 4^{-n(n+1)} 8^{-n-1} K_0^{-n-1} \delta_0 \in (0,1),
\end{equation}
and let $\overline{\epsilon}_0 =\overline{\epsilon}_0(n, \nu, K_0)>0$, $\bar \epsilon_2 =\bar \epsilon_2(n,\nu, K_0, \frac{1}{16}) > 0$ and $\epsilon_3 = \epsilon_3(n,\nu, K_0, 1-\delta_1) > 0$  respectively be the constants defined in Corollary~\ref{cor-1}, Corollary~\ref{prop-up-lemma} and Proposition~\ref{third-growth}. Now, we define
\begin{equation}\label{choice of epsilon}
\epsilon=\epsilon(n,\nu, K_0)
=\min \{\epsilon_0, \overline{\epsilon}_0, \bar \epsilon_2, \epsilon_3 \} >0,
\end{equation}
and we prove the assertion of the theorem with this choice of $\epsilon$.

\smallskip
By performing a linear translation, we can assume  $Y=(0,0)$.  As in the proof of Corollary \ref{cor-1}, we define 
\begin{align*}
d(X) = 
\begin{cases} 
\sup \left\{ \rho > 0 : C_{\rho, \mu}(X) \subset C_{2r, \mu} \right\} & \text{for } X = (x, t) \in C_{2r, \mu}, \\
0 & \text{for } X=(x,t) \in \partial_p C_{2r, \mu}.
\end{cases}
\end{align*}
Also, recall that 
\[
U_1 =B_{r}\times
\big( -3r^2(\mu)_{B_{2r}}^{1/n}, -2r^2(\mu)_{B_{2r}}^{1/n} \big), 
\]
and by \eqref{cylinder measure formula}
\[
\left(\frac{r}{2}\right)^2(\mu)_{B_{r/2}(x)}^{1/n}\leq r^2(\mu)_{B_{2r}}^{1/n}\qquad \text{for all} \quad x \in B_{r}.
\] 
We then see that  
\begin{equation}\label{dist in C^1}
\frac{r}{2}\leq d(X) \leq  2r \quad \text {for all}\quad X\in \overline{U_1}.
\end{equation}
We denote the lower half cylinder of $C_{2r,\mu}$ by
\[
Q^0=B_{2r}\times\big(-4r^2(\mu)_{B_{2r}}^{1/n},-2r^2(\mu)_{B_{2r}}^{1/n}\big),
\]
and
\[
M=\sup_{X\in Q^0}\big(d^{\gamma}(X)u(X)\big) \geq 0.
\]
Here, the same as in \eqref{epsilon using 1}, $\gamma >0$ is the constant defined in \eqref{gamma-def-0908} which comes from the constant defined in Corollary~\ref{prop-up-lemma}.  

\smallskip
From now on, we assume that $M>0$, as otherwise, it follows from the maximum principle that $u=0$ in $C_{2r, \mu}$ and then \eqref{Harnack} is trivial.  Because $M>0$ and because of the continuities of $u$ and $d$, we can find $X_0=(x_0, t_0) \in  \overline{Q^0}$ such that 
\[ M=d^{\gamma}(X_0)u(X_0) \quad \text{and} \quad d^{\gamma}(X_0)>0.
\]
Because $U_1 \subset Q^0$, it follows from the definition of $M$ and \eqref{dist in C^1} that
\begin{equation} \label{sup u-part 1}
\sup_{U_1}u\leq 2^{\gamma}r^{-\gamma} M =\Big(\frac{2d(X_0)}{r}\Big)^{\gamma}u(X_0).
\end{equation} 
Next, let us denote
\begin{align*}
& \rho_0=\frac{d(X_0)}{32} \in(0,\frac{1}{16}r] , \quad
C_0 =  C_{\rho_0,\mu}(X_0), \quad \text{and} \\
& \Gamma_0 = \Big\{ u > \frac{1}{2} u(X_0) \Big\} \cap C_0 .
\end{align*}
We observe that $C_0 \subset C_{16\rho_0, \mu}(X_0)$. Then, for any $Z \in C_0$, it follows from Lemma \ref{cylinder-inclusion} that
\[ 
C_{16\rho_0, \mu}(Z) \subset C_{32\rho_0, \mu}(X_0) \subset C_{2r, \mu},
\]
which implies $d(Z) \geq 16\rho_0 = d(X_0)/2$. From this and as $C_0 \subset Q^0$, it follows that
\begin{equation}\label{u-c-0}
\sup_{C_0} u \leq 2^{\gamma} u(X_0).
\end{equation}
Given the choice of $\epsilon$ in \eqref{choice of epsilon}, we apply the contrapositive of claim \eqref{claim-1} to $u - u(X_0)/2$ in $C_0$, which yields the inequality
\begin{equation}\label{v measure-1}
\mu (\Gamma_0) \geq \delta_0 \mu(C_0),
\end{equation}
where $\delta_0 = \delta_0(n, \nu, K_0)\in (0,1)$ is defined in \eqref{epsilon using 1}. Indeed, if \eqref{v measure-1} was not true, then from \eqref{claim-1}, \eqref{epsilon using 1}, and \eqref{sup u-part 1}, we would infer that
\[
\frac{u(X_0)}{2} \le 2^{-\gamma-1} \sup_{C_0} \Big(u - \frac{u(X_0)}{2}\Big) < 2^{-\gamma-1} \sup_{C_0} u,
\]
which contradicts \eqref{u-c-0}.

\smallskip
Now, let $Y_0 = \big(x_0, t_0 + 48\rho_0^2 (\mu)_{B_{8\rho_0}(x_0)}^{1/n}\big)$, and
\[
Q^1 = B_{4\rho_0}(x_0) \times \big(t_0 - 8\rho_0^2 (\mu)_{B_{8\rho_0}(x_0)}^{1/n}, t_0\big).
\]
By a straightforward calculation using \eqref{cylinder measure formula} and the doubling property of $\mu$ stated in \eqref{doubling property}, we find that
\[
C_0 = C_{\rho_0,\mu}(X_0) \subset Q^1 \quad \text{and} \quad \mu(Q^1) \leq 4^{n(n+1)}8^{n+1} K_0^{n+1} \mu(C_0).
\]
Therefore, from this, \eqref{def of delta-1}, and \eqref{v measure-1}, we see that
\[
\mu\Big(\Big\{\frac{2u}{u(X_0)} \geq 1\Big\} \cap Q^1\Big) 
\geq \mu(\Gamma_0) \geq \delta_0 \mu(C_0)  
\geq \delta_1 \mu(Q^1).
\]
Note that this is the same as \eqref{corol-0907-cond} when $\delta= 1-\delta_1$, $\frac{2u}{u(X_0)}$ is in place of  $u$, and $Q^1$ is in place of $Q$.  Because of this and the choice of $\epsilon$ in \eqref{choice of epsilon}, we apply Corollary \ref{Growth corollary} to $\frac{2u}{u(X_0)}$ in $C_{8\rho_0,\mu}(Y_0)$ to obtain
\begin{equation}\label{sup u-part 2}
u \geq \beta u(X_0) \quad \text{in} \quad C_{4\rho_0,\mu}(Y_0), \quad \text{where} \quad \beta = \beta(n, \nu, K_0) \in (0,1).
\end{equation}

\smallskip
Next, as $t_0\in \big(-4r^2(\mu)_{B_{2r}}^{1/n},-2r^2(\mu)_{B_{2r}}^{1/n}\big]$ and $\rho_0\in(0,\frac{r}{16}]$, it follows that
\begin{align} \label{tau def}
\tau:=t_0+48{\rho_0}^2(\mu)_{B_{8\rho_0}(x_0)}^{1/n}\in\Big(-4r^2(\mu)_{B_{2r}}^{1/n},-\frac{5}{4}r^2(\mu)_{B_{2r}}^{1/n}\Big).
 \end{align}
From this, and the choice of $\epsilon$ in \eqref{choice of epsilon}, we apply Corollary~\ref{prop-up-lemma} to $u$ in $C_{2r,\mu}$ with $z=x_0$, $h=\frac{1}{16}$, $4\rho_0$ in place of $\rho$, and $2r$ in place of $r$. We obtain
\begin{equation}\label{sup u-part 3}
 \inf_{B_{4\rho_0}(x_0)} u(\cdot, \tau)\leq \Big(\frac{2r}{\rho_0}\Big)^{\gamma} \inf_{B_{r}} u(\cdot, \sigma)\quad \text{for all}\quad \sigma\in \big(-r^2(\mu)_{B_{2r}}^{1/n},0\big),
\end{equation}
where $\gamma>0$ is defined in \eqref{gamma-def-0908}. 

\smallskip
Now, as $d(X_0)=32\rho_0$, we infer from \eqref{sup u-part 1} that
\[
\sup_{U_1}u \leq \Big( \frac{64\rho_0}{r} \Big)^{\gamma} u(X_0).
\]
From this, and by \eqref{sup u-part 2}, \eqref{tau def}, and the definition of $Y_0$, it follows that
\[
\sup_{U_1}u \leq \Big( \frac{64\rho_0}{r} \Big)^{\gamma} \beta^{-1} \inf_{B_{4\rho_0}(y_0)}u(\cdot, \tau).
\]
Then, using \eqref{sup u-part 3}, we obtain
\[
\sup_{U_1}u \leq 128^{\gamma}\beta^{-1}\inf_{U_2}u.
\]
This implies \eqref{Harnack} with $N=N(n,\nu, K_0)=128^{\gamma}\beta^{-1}>1$. The proof is completed. 
\end{proof}
\smallskip
\subsection{Proof of Corollary~\protect{\ref{Holder regularity}}} We adopt the argument from \cite{DeGiorgi, Moser-1, Moser-2}. 
The main idea  is to use the Harnack inequalities to derive oscillation decay estimates for $u$ on a nested sequence of weighted cylinders.
\begin{proof} Without loss of generality that $Y=0$. We first prove  \eqref{L-infty-est-intro}. Take any $Z\in C_{3r/2,\mu}$, by Lemma \ref{cylinder-inclusion}, we have $C_{r/2,\mu}(Z)\subset C_{2r,\mu}$. It then follows from the choice of $\epsilon$ in \eqref{choice of epsilon}, Corollary \ref{cor-1}, and the doubling property of $\mu$ in Lemma \ref{A-p doubling} that
\begin{align*}
u^+(Z) & \leq N \Big(\frac{1}{\mu(C_{r/2, \mu}(Z))} \int_{C_{r/2, \mu}(Z)} u^+(x, t)^q \mu(x)\, dxdt\Big)^{1/q} \\
& \leq N \Big(\frac{1}{\mu(C_{2r, \mu})} \int_{C_{2r, \mu}} |u(x, t)|^q \mu(x)\, dxdt\Big)^{1/q},
\end{align*}
where $N=N(n,\nu, K_0,q)>0$. Note that the same estimate also holds for $-u$. Then, since $Z$ is arbitrary in $C_{3r/2,\mu}$, we conclude that
\begin{equation}  \label{L3r/2-est-0905}
\|u\|_{L^\infty(C_{3r/2,\mu})} \leq N \Big(\frac{1}{\mu(C_{2r, \mu})} \int_{C_{2r, \mu}} |u(x, t)|^q \mu(x)\, dxdt\Big)^{1/q}.
\end{equation}
This implies  \eqref{L-infty-est-intro}.

\smallskip
Next, we prove \eqref{Holder}. By the scaling $u \mapsto u/\lambda$ with suitable $\lambda>0$, we can assume without loss of generality that
\begin{equation} \label{l-q-norm-small}
 N \Big(\frac{1}{\mu(C_{2r, \mu})} \int_{C_{2r, \mu}} |u(x, t)|^q \mu(x)\, dxdt\Big)^{1/q} \leq \frac{1}{2},
\end{equation}
where $N>0$ is defined in \eqref{L3r/2-est-0905}. Therefore, \eqref{Holder} is reduced to
\begin{equation} \label{reduced-holder-0906}
|u(X) - u(X_0)| \leq N r^{-\alpha} \rho_\omega(X, X_0)^{\alpha}, \quad \forall \ X, X_0 \in C_{r, \mu},
\end{equation}
with some generic constant $N>0$ and some $\alpha \in (0,1]$. 

\smallskip
To prove \eqref{reduced-holder-0906}, let  us fix $X_0 =(x_0, t_0) \in C_{r,\mu}$ and $X = (x, t) \in C_{r, \mu}$.Without loss of generality, we  only need to consider the case that $t \leq t_0$.  We then denote
\[ 
M(\tau)=\sup_{C_{\tau,\mu}(X_0)}u\quad\text{and}\quad m(\tau)=\inf_{C_{\tau,\mu}(X_0)}u , \quad 
\] 
for all $\tau>0$ such that $C_{\tau,\mu}(X_0)\subset C_{3r/2,\mu}$. The oscillation of $u$ in $C_{\tau,\mu}(X_0)$ is defined by
\[
\phi(\tau)=M(\tau)-m(\tau)\quad \text{for}\quad \tau \in[0, \tau_0),
\]
where  $\tau_0=\tau_0(X_0) \in [ r /2 , 3r / 2 ]$. We note that $\phi$ is a non-decreasing function. For simplicity, we also define some sub-cylinders in $C_{\tau,\mu}(X_0)$ 
by 
\begin{align*} 
C_{\tau,\mu}^{+}(X_0)&=B_{\tau/2}(x_0)\times \Big(t_0-\frac{1}{4}\tau^2(\mu)_{B_\tau(x_0)}^{1/n}, t_0\Big),\\
C_{\tau,\mu}^{-}(X_0)&=B_{\tau/2}(x_0)\times \Big(t_0-\frac{3}{4}\tau^2(\mu)_{B_\tau(x_0)}^{1/n}, t_0-\frac{1}{2}\tau^2(\mu)_{B_\tau(x_0)}^{1/n}\Big).
\end{align*}
It is not hard to check that 
\begin{equation}\label{sub-cylinder inclu}
C_{\tau/4,\mu}(X_0)\subset C_{\tau,\mu}^{+}(X_0).
\end{equation}
As $X_0\in C_{r,\mu}$, it follows from Lemma \ref{cylinder-inclusion} that there exists $r_0=r_0(X_0)\geq \frac{r}{2}$ such that $C_{r_0,\mu}(X_0)\subset C_{3r/2,\mu}$.
Also, note that
 \[
M(r_0)-u(x)\geq 0, \quad u(x)-m(r_0)\geq 0\quad \text{in}\quad  C_{r_0,\mu}(X_0) .
\]
We apply Theorem \ref{Harnack inequality} to $M(r_0)-u(x)$ and $u(x)-m(r_0)$ and by \eqref{sub-cylinder inclu}, then
\begin{equation}\label{M-u}
\begin{aligned}
\sup_{C^-_{r_0,\mu}(X_0)}\big(M(r_0)-u\big)&\leq N\inf_{C^+_{r_0,\mu}(X_0)}\big(M(r_0)-u\big)\\
&\leq N\inf_{C_{r_0/4,\mu}(X_0)}\big(M(r_0)-u\big),
\end{aligned}
\end{equation}
and
\begin{equation}\label{u-m}
\begin{aligned}
\sup_{C^-_{r_0,\mu}(X_0)}\big(u-m(r_0)\big)&\leq N\inf_{C^+_{r_0,\mu}(X_0)}\big(u-m(r_0)\big)\\
&\leq N\inf_{C_{r_0/4,\mu}(X_0)}\big(u-m(r_0)\big),
\end{aligned}
\end{equation}
where $N=N(n,\nu, K_0)>1$. 
Notice also that 
$m(r_0)\leq u(x)\leq M(r_0)$ for all $x\in C^-_{r_0,\mu}(X_0)$, and 
\[
M(4^{-1}r_0)=\sup_{C_{r_0/4,\mu}(X_0)}u,\quad\quad m(4^{-1}r_0)=\inf_{C_{r_0/4,\mu}(X_0)}u.
\] 
Therefore, adding both sides of \eqref{M-u} and \eqref{u-m}, we obtain
\[
N\left\{M(4^{-1}r_0)-m(4^{-1}r_0)\right\}\leq (N-1)\left\{M(r_0)-m(r_0)\right\}.
\]
Note that due to \eqref{L3r/2-est-0905} and \eqref{l-q-norm-small}, we have  $\phi(r_0)\leq 1$, hence,
\begin{equation*}
\phi(4^{-1}r_0)\leq \frac{N-1}{N}\phi(r_0)\leq \frac{N-1}{N}.
\end{equation*}
Then, by induction, for $\tau=4^{-k}r_0$ with $k\in \N$, we have the following oscillation decay estimate
\begin{equation}\label{oscillation decay}
\phi\big(4^{-k}r_0\big)\leq \frac{N-1}{N}\phi\big(4^{-k+1}r_0\big)\leq\cdots\leq \Big(\frac{N-1}{N}\Big)^k.
\end{equation}
We now consider the following two cases.

\smallskip \noindent
{\bf Case 1}. $X=(x,t)\in C_{r_0,\mu}(X_0)$. Then, there exists a unique integer $k_0\in \N$ such that 
\begin{equation}\label{X position}
X\in C_{4^{-k_0}r_0,\mu}(X_0)\setminus C_{4^{-k_0-1}r_0,\mu}(X_0).
\end{equation}
Now, recall the definition of quasi-metric $\rho_{\omega}$ in \eqref{quasi-metric-intro}, we have
\[
\begin{split}
\rho_{\omega}(X,X_0)^2 & =\left[\max\big\{|x-x_0|, \phi_{x_0}^{-1}(t_0-t)\big\}\right]^2 \\
& \quad +\min\big\{|x-x_0|^2, (\mu)_{B_{\Theta_{X_0}}(x_0)}^{-1/n}\cdot(t_0-t)\big\},
\end{split}
\]
where $\phi_{x_0}$ and $\Theta_{X_0}$ are defined in \eqref{phi-y} and \eqref{theta-introc}, respectively. By \eqref{X position}, we have 
\begin{equation*}%\label{rho theta}
4^{-k_0-1}r_0 \leq \Theta_{X_0}(X)< 4^{-k_0}r_0.  
\end{equation*}
Also, from \eqref{d-r}, we have
\begin{equation*}
\Theta_{X_0}(X)\leq \rho_{\omega}(X,X_0)\leq \sqrt {2} \Theta_{X_0}(X).
\end{equation*}
 Then, combining the last two estimates with the fact that $r_0\geq \frac{r}{2}$, we obtain
\begin{equation}\label{quasi dist X-X_0}
4^{-k_0-1}2^{-1}r\leq 4^{-k_0-1}r_0\leq \rho_{\omega}(X, X_0).
\end{equation}
On the other hand, by \eqref{oscillation decay}, we have
\begin{equation}\label{u-difference}
|u(X)-u(X_0)|\leq \phi(4^{-k_0}r_0)\leq \Big(\frac{N-1}{N}\Big)^{k_0}=4^{-\alpha k_0}
\end{equation}
with $\alpha=-\log_4(1-1/N)\in (0,1]$, this is because we can always choose $N\geq \frac{4}{3}$ so that $\alpha\in (0,1]$. 
Using \eqref{quasi dist X-X_0} and \eqref{u-difference}, we have
\[
|u(X)-u(X_0)| \leq 8^{\alpha}r^{-\alpha}  \rho_{\omega}(X, X_0)^{\alpha}.
\]
This implies \eqref{reduced-holder-0906}.

\smallskip \noindent
{\bf Case 2}. $X=(x,t)\in C_{r,\mu}\setminus C_{r_0,\mu}(X_0)$. As  $t\leq t_0$, we have
\[
\rho_{\omega}(X, X_0)\geq r_0\geq \frac{r}{2}. 
\]
Also, from \eqref{L3r/2-est-0905} and \eqref{l-q-norm-small}, we have $|u(X) - u(X_0)| \leq 1$. Therefore,
\[ 
|u(X)-u(X_0)| \leq 2^{\alpha}r^{-\alpha} \rho_{\omega}(X, X_0)^{\alpha}.
\]
This also implies \eqref{reduced-holder-0906}. The proof of the corollary is completed.
\end{proof}
\subsection{Proof of Corollary~\protect{\ref{Liouville theorem}}}
\begin{proof}
We use the idea in the proof of Corollary~\ref{Holder regularity}. By subtracting a suitable constant, we may assume that $u$ is non-negative. Then, for any $r >0$, let 
\[ 
M(r) = \sup_{C_{r, \mu}} u , \quad 
m(r) = \inf_{C_{r, \mu}} u , \quad 
\phi(r)= M(r) -m(r). 
\]
Following the proof of Corollary~\ref{Holder regularity}, 
we have, for any $k \in \mathbb{N}$,  
\[
\phi(r)  \le \Big(\frac{N-1}{N}\Big)^k \phi(4^k r),\quad \text{where} \quad N = N(n, \nu, K_0)>1.
\]
Since $u$ is bounded and $r$ is an arbitrary positive number, it implies that $\phi(r) =0$. Therefore, $u$ is a constant. The proof of the corollary is completed.
\end{proof}
%=====
\appendix 
%=====
\section{Proof of Lemma~\ref{covering-1}}  \label{Appendix-A}
%\subsection{Proof of \eqref{conclusion covering}}
\begin{proof}[Proof of Lemma~{\protect\ref{covering-1}}] We begin by proving the first assertion in \eqref{conclusion covering} which claims that
\[
\mu(\Gamma\setminus E) =0.
\] 
Due to $\omega\in A_{1+\frac{1}{n}}$, we have $\mu(x)=\omega(x)^{-n}>0$ for almost every $x\in \R^n$. Consequently, one can verify that the $\mu$-density of $\Gamma$ is 1 at almost every point of $\Gamma$. This, together with the definition of $E$ in \eqref{set E}, implies that 
\[
|\Gamma\setminus E|=0.
\]
Thus, the claim above also holds true. The first assertion in  \eqref{conclusion covering} is proved.

\smallskip
Now,  we prove the second assertion in \eqref{conclusion covering}. Note that this assertion is trivial when $\mu(\Gamma) =0$. Hence, from now on, we assume that $\mu(\Gamma)>0$.  We now define
\begin{equation}\label{redef of A}
\tilde \A= \Big\{C = C_{r, \mu}(Y): \mu(C \cap \Gamma) = (1-\delta_0) \mu(C) \Big\}.
\end{equation}
It is clear that $\tilde{\mathcal{A}} \subset \A$. Moreover, for a fixed $C_{r,\mu}(Y)\in \A$ such that
\[
\mu(C_{r,\mu}(Y) \cap \Gamma) \geq (1-\delta_0) \mu(C_{r,\mu}(Y)),
\]
let us define $C_\theta = C_{\theta r,\mu}(Y)$, and
\[
g(\theta) = \frac{1}{\mu(C_{\theta})} \int_{C_{\theta}  \cap \Gamma} \mu(x) dx dt.
\] 
It follows that $g: [1, \infty) \rightarrow (0, \infty)$ is continuous. Moreover,   
\[
 g(1) \geq 1-\delta_0 \quad \text{and}  \quad \lim_{\theta \rightarrow \infty}g(\theta) =0,
\]
where the latter is due to the fact that $\Gamma$ is bounded 
and thus has a finite measure. 
Hence, we can find $\theta_0 \geq 1$ so that 
$g(\theta_0) = 1-\delta_0$ by the continuity of $g$. 
This implies that $C_{\theta_0} \in \tilde{\mathcal{A}}$. Also, as $C_{r,\mu}(Y) \subset C_{\theta_0}$, we conclude that  
\[
E = \displaystyle{\bigcup_{C \in \mathcal{A}}} C = \displaystyle{\bigcup_{C \in \tilde{\mathcal{A}}}} C.
\]

Next, we construct a sequence of pairwise disjoint weighted cylinders $\{C^i\}_{i=0}^{\infty}$ from $\tilde{\mathcal{A}}$ as in the proof of the classical Vitali covering lemma. 
First of all, note that $\Gamma$ is bounded, we see that 
the set $\{r>0: C_{r,\mu}(Y)\in \tilde \A\}$ is also bounded. 
Let us define $\A_0 =   \tilde \A$, and
\[
R_0=\sup\{r>0: C_{r,\mu}(Y)\in \tilde \A\}<\infty.
\]
Then, by a compactness argument, we can find $C^0$ such that 
\[ 
C^0 = C_{R_0,\mu}(Y^0)  \in   \A_0. 
\]
Next, for each $i \in \N$, by a similar argument, we can define $\A_i, R_i, C^i$ inductively by
\[
\A_i=\Big\{C=C_{r,\mu}(Y)\in \tilde \A: C\cap C^k=\emptyset, \ \forall \ k =0,1,\ldots ,i-1\Big\}
\]
and
\[
R_i =\sup\{r>0 : \ C_{r,\mu}(Y)\in \A_i\}, \quad C^i =C_{R_i,\mu}(Y^i)  \in   \A_i. 
\] 
If the set $\A_i$ is non-empty, we can continue this process. If the set $\A_i$ is empty, then stop the process and set $R_{k}=0$ for all $k\geq i$. 

\smallskip
From the construction, we see that
\[
\tilde \A= \A_0\supset\A_1\supset\A_2\supset \cdots \quad \text{and}\quad R_0\geq R_1\geq R_2\geq R_3\geq\cdots .
\]
We claim that 
\begin{equation}  \label{R-k-limit}
\lim_{k\rightarrow \infty} R_k =0.
\end{equation}
We prove \eqref{R-k-limit} by using a contradiction argument. Suppose that \eqref{R-k-limit} does not holds. Then, as $\{R_k\}_{k}$ is monotone decreasing, there is $r_0>0$ such that
\[ 
\lim_{k\rightarrow \infty} R_k =r_0 \quad \text{and} \quad R_k \geq r_0, \quad \forall \ k \in \mathbb{N}.
\]
Since $\Gamma$ is bounded, 
we can choose  $R>0$ sufficiently large so that
\[
x \in B_R \quad \text{if} \quad (x,t) \in \Gamma.
\]
We claim that there is a constant $N=N(n, K_0, R_0, R, r_0)>0$ such that
\begin{equation} \label{C-i-measure}
\mu(C^i)\geq Nr_0\mu(B_R)^{(n+1)/n}, \quad \forall \ i\in \N.
\end{equation}
Note that as $\{C^i\}$ are pairwise disjoint, \eqref{C-i-measure} implies a contradiction because 
\[
\begin{split}
\mu(\Gamma) &\geq \mu\big(\Gamma\cap(\cup^{\infty}_{i=0}C^i)\big)=\sum^{\infty}_{i=0}\mu\big(\Gamma\cap C^i\big) = (1-\delta_0)\sum^{\infty}_{i=0}\mu(C^i) \\
& \geq (1-\delta_0) \sum^{\infty}_{i=0}Nr_0\mu(B_R)^{(n+1)/n} =\infty , 
\end{split}
\]
and $\mu(\Gamma) <\infty$. Hence,  \eqref{R-k-limit} is proved once we verify \eqref{C-i-measure}. To this end, for each $i =0, 1, \ldots$, we write $C^{i}=C_{R_i,\mu}(Y^i) \in \mathcal{A}_i$  with $Y^i = (y_i, s_i) \in \mathbb{R}^{n+1}$. Due to the definition of $\tilde \A$, and the fact that $\mathcal{A}_i \subset \tilde \A$, we see that $B_R\cap B_{R_i}(y^i)\neq \emptyset$.  Then, if $B_R\subset B_{R_i}(y^i)$, by formula \eqref{cylinder measure formula}, we have
\begin{equation}\label{case 1}
\mu(C^i) = \sigma_n^{-1/n} R_i \mu(B_{R_i}(y^i))^{(n+1)/n} \geq  
\sigma_n^{-1/n} r_0{\mu\big(B_R\big)}^{(n+1)/n}.
\end{equation}
Otherwise, we see that $B_R \setminus B_{R_i}(y^i)\neq \emptyset$. Then, let $z \in B_R\cap B_{R_i}(y^i)$, it follows that
\[ |y^i|\leq |z| + |z-y^i| \leq R+R_i\leq R+R_0.
\]
Now, let $\tilde R= 2R+R_0$, and we see that
\[
B_R\subset B_{\tilde R}(y^i).
\]
Also, as $\mu \in A_{n+1}$, it follows from the doubling property of the $A_{n+1}$ weight (Lemma \ref{A-p doubling}) that
\begin{align*}
\mu(B_R)\leq\mu(B_{\tilde R}(y^i)) 
\leq K_0^n\Big( \frac{2R+ R_0}{r_0} \Big)^{n(n+1)}\mu(B_{R_i}(y^i)),
\end{align*}
which implies
\[
\mu(B_{R_i}(y^i)) \geq N \mu(B_R), \quad \text{where}\quad N=N(n,K_0, R_0, R, r_0)>0.
\]
Then
\begin{equation}\label{case 2}
\mu(C^i) = \sigma_n^{-1/n} R_i \mu(B_{R_i}(y^i))^{(n+1)/n} \geq Nr_0{\mu\big(B_R\big)}^{(n+1)/n}.
\end{equation}
%with some $N=N(n,K_0, R_0, R, r_0)>0$. 
Hence, \eqref{C-i-measure} follows from \eqref{case 1} and \eqref{case 2}, which also completes the proof of the assertion \eqref{R-k-limit}.

\smallskip
Next, for each $i\in \N$, we denote 
\[
\tilde {C}^{i}=B_{3R_{i}}(y^{i})\times \big(s_i-6R_i^2(\mu)^{1/n}_{B_{3R_{i}}(y^{i})},s_i+3R_i^2(\mu)^{1/n}_{B_{3R_{i}}(y^{i})}\big).
\]
Then, for any $C_{r,\mu}(Y) \in \A$, by \eqref{R-k-limit} and the construction of the series $\{R_k\}_k$, we can find a unique $k_0 \in \mathbb{N}$ such that
\begin{equation} \label{C-k-zero}
R_0\geq R_1\geq \ldots \geq R_{k_0} \geq r> R_{k_0+1}\geq R_{k_0+2} \geq \ldots.
\end{equation}
From this and the construction of $\{C^i\}$, we infer that there is some $i_0\in\{0,1,2,\cdots, k_0\}$ such that
\begin{equation} \label{i-zero-C}
C_{r,\mu}(Y)\cap C^{i_0} \not=\emptyset.
\end{equation}
Because of \eqref{C-k-zero} and \eqref{i-zero-C}, we see that  $C_{r,\mu}(Y) \subset \tilde C^{i_0}$. From this, the first assertion in \eqref{conclusion covering}, and by ignoring the set of zero $\mu$-measure, we have
\begin{equation} \label{Gamma-comp-Ci}
\Gamma\subset E\subset \bigcup^{\infty}_{i=0}\tilde C^i.
\end{equation}
From this, and by applying Lemma \ref{A-p doubling}, we get 
\[
\mu\big(\tilde C^{i}\big) \leq 3^{(n+1)^2+1}K_0^{n+1}\mu(C^{i}).
\]
Now, due to the pairwise disjoint property of $\{C^i\}$ and \eqref{redef of A}, we have
\[
\mu(E\setminus \Gamma)\geq\sum^{\infty}_{i=0}\mu\big((E\setminus \Gamma)\cap C^i\big)=\delta_0\sum^{\infty}_{i=0}\mu(C^i).
\]
From the last two estimates, and \eqref{Gamma-comp-Ci}, we  infer that
\begin{align*}
\mu(E) & =\mu(\Gamma)+ \mu(E\setminus \Gamma) \geq \mu(\Gamma)+\delta_0\sum^{\infty}_{i=0}\mu(C^i)
\\
&\geq \mu(\Gamma)+3^{-(n+1)^2-1}K_0^{-n-1}\delta_0\sum^{\infty}_{i=0}\mu(\tilde C^i)
\\
&\geq \mu(\Gamma)+3^{-(n+1)^2-1}K_0^{-n-1}\delta_0\mu(\Gamma)
\\
&=\big(1+3^{-(n+1)^2-1}K_0^{-n-1}\delta_0\big)\mu(\Gamma).
\end{align*}
Then, by setting $q_0=1+3^{-(n+1)^2-1}K_0^{-n-1}\delta_0>1$, we obtain the second assertion in \eqref{conclusion covering}. 
% The proof of \eqref{conclusion covering} is completed.
%\subsection{Proof of lemma 2.8}
%\begin{proof}[Proof of \eqref{hat-E-E}] 

\smallskip
Now, we prove the last claim of the Lemma, \eqref{hat-E-E}. 
For each $x \in \mathbb{R}^n$, we denote
\[
E(x) =\big\{ t\in \mathbb{R}: (x,t) \in E \big\} \quad \text{and}\quad \hat{E}(x) =\big\{t \in \mathbb{R}: (x,t) \in \hat{E} \big\}.
\]
By the Fubini theorem, we have
\[
\mu(E) = \int_{\mathbb{R}^n} |E(x)| \mu(x) dx \quad \text{and} \quad \mu(\hat{E}) = \int_{\mathbb{R}^n} |\hat{E}(x)| \mu(x) dx,
\]
where $|E(x)|$ and $|\hat{E}(x)|$ denote the Lebesgue measure  on $\mathbb{R}$ of $E(x)$ and $\hat{E}(x)$, respectively.
As $\mu(x) \geq 0$, in order to prove \eqref{hat-E-E}, it suffices to show
\begin{equation}\label{time measure}
 |\hat{E}(x)| \geq q_1 |E(x)| \quad \text{for all}\ x \in \mathbb{R}^n, \text{ where } q_1 = (K_1-1)(K_1+1)^{-1}.
\end{equation}
%where $q_1$ is defined in \eqref{hat-E-E}. 
For any fixed $x$ with $E(x)$ non-empty, we have $\hat E(x)\subset \R$ is non-empty and open, thus we may assume $\hat E(x)$ is a finite or countable union of disjoint open intervals, i.e.,
\[
\hat E(x)=\cup_{i\in \mathcal{I}}I_i(x),
\]
where $\mathcal {I}$ is an index set. Then, for any $t\in E(x)$, we have $(x,t)\in C_{\tilde r,\mu}(\tilde Y)\in \A$ with some $\tilde r>0$ and $\tilde Y=(\tilde y, \tilde s)\in \R^{n+1}$. Due to \eqref{set hat E}, we see 
\[
\big(\tilde s+{\tilde r}^2(\mu)^{1/n}_{B_{\tilde r}(\tilde y)},\tilde s+K_1{\tilde r}^2(\mu)^{1/n}_{B_{\tilde r}(\tilde y)}\big)\subset I_{i_0}(x)\quad \text{for some}\quad i_0\in \mathcal {I}.
\]
Set $r_{i_0}^2(x)=|I_{i_0}(x)|/(K_1-1)$, we can determine $a_{i_0}\in \R$ depending only on $x$ so that
\[
I_{i_0}(x)=(a_{i_0}+r_{i_0}^2, a_{i_0}+K_1r_{i_0}^2).
\]
Also, note that $r_{i_0}^2\geq {\tilde r}^2(\mu)^{1/n}_{B_{\tilde r}(\tilde y)}$, it follows from the last two formulas that
\begin{equation}\label{pull back}
a_{i_0}-r_{i_0}^2\leq \tilde s-{\tilde r}^2(\mu)^{1/n}_{B_{\tilde r}(\tilde y)}\leq \tilde s\leq a_{i_0}+K_1r_{i_0}^2.
\end{equation}
For convenience, we set $J_{i}(x)=(a_{i}-r_{i}^2,a_{i}+K_1r_{i}^2)$ for all $i\in \mathcal {I}$.  By \eqref{pull back}, we have
$E(x)\subset \bigcup_{i\in \mathcal {I}}J_i(x)$ and 
\[
|J_i(x)|=(K_1+1)r_i^2=\frac{K_1+1}{K_1-1}|I_i(x)|\quad \text{for all}\quad  i\in \mathcal {I}.
\]
 Thus, for all $x\in \R^n$ with $E(x)$ non-empty, we have 
\begin{equation*}
|\hat E(x)|=\sum_{i\in \mathcal {I}}|I_i(x)|= \sum_{i\in \mathcal {I}}\frac{K_1-1}{K_1+1}|J_i(x)|\geq \frac{K_1-1}{K_1+1}|E(x)|.
\end{equation*}
Taking $q_1=(K_1-1)(K_1+1)^{-1}$, we proved \eqref{time measure}. 
\end{proof}
%\nocite{*}
%===========

\end{document}